\documentclass[a4paper,11pt]{amsart}
\usepackage{enumerate}
\usepackage{array}
\usepackage{amsmath, amssymb, amsfonts}
\usepackage{booktabs}
\usepackage[table]{xcolor}
\usepackage{cite}
\usepackage{tikz-cd,pgffor}
\usepackage{footnote}
\usepackage{color,xcolor}
\usetikzlibrary{positioning}
\usepackage[poly,all]{xy}
\usepackage{dynkin-diagrams}
\usepackage{graphicx}
\usepackage{subfigure}
\usepackage{mathrsfs}

\allowdisplaybreaks[3]
\numberwithin{equation}{section}

\usepackage[backref=page,colorlinks=true, pdfstartview=FitV,linkcolor=blue,citecolor=red,urlcolor=blue]{hyperref}
\theoremstyle{plain}
\newtheorem{theorem}{Theorem}[section]
\newtheorem{lemma}[theorem]{Lemma}
\newtheorem{remark}[theorem]{Remark}
\newtheorem{corollary}[theorem]{Corollary}
\newtheorem{definition}[theorem]{Definition}
\newtheorem{proposition}[theorem]{Proposition}
\newtheorem{example}[theorem]{Example}

\foreach \x in {A,...,Z}{%
	\expandafter\xdef\csname cal\x\endcsname{\noexpand\ensuremath{\noexpand\mathcal{\x}}}
}

\foreach \x in {A,...,Z}{%
	\expandafter\xdef\csname bb\x\endcsname{\noexpand\ensuremath{\noexpand\mathbb{\x}}}
}

\foreach \x in {a,...,z}{%
	\expandafter\xdef\csname frak\x\endcsname{\noexpand\ensuremath{\noexpand\mathfrak{\x}}}
}

\foreach \x in {A,...,Z}{%
	\expandafter\xdef\csname scr\x\endcsname{\noexpand\ensuremath{\noexpand\mathscr{\x}}}
}

\foreach \x in {A,...,Z}{%
	\expandafter\xdef\csname rm\x\endcsname{\noexpand\ensuremath{\noexpand\mathrm{\x}}}
}

\newcommand{\catO}{\calO}

\newcommand{\Hom}{\mathrm{Hom}}
\newcommand{\End}{\mathrm{End}}

\newcommand{\Span}{\mathrm{span}}

\newcommand{\Res}{\mathrm{Res}}
\newcommand{\Ind}{\mathrm{Ind}}
\newcommand{\Spec}{\mathrm{Spec}}
\newcommand{\inprod}[2]{({#1, #2^\vee})}

\newcommand{\pr}{\mathrm{pr}}
\newcommand{\ch}{\mathrm{ch}}
\newcommand{\height}{\mathrm{ht}}

\newcommand{\rmp}{\mathrm{p}}

\newcommand{\Sing}{\mathrm{Sing}}
\newcommand{\Tor}{\mathrm{Tor}}
\newcommand{\val}{\mathrm{val}_T}
\newcommand{\bfk}{\mathbf{k}}

\newcommand{\zza}{Z_{\zeta,0}}

\newcommand{\afun}{\mathbf{a}}
\newcommand{\Wch}{\widetilde{W}}
\newcommand{\Stab}{\mathrm{Stab}}
\newcommand{\bfF}{\mathbf{F}}

\newcommand{\Ki}{K_{\alpha_i}}

\newcommand{\inv}{^{-1}}
\newcommand{\widesim}[2][1.5]{
	\mathrel{\overset{#2}{\scalebox{#1}[1]{$\sim$}}}
}

\usepackage{lineno}

\begin{document}
	
	\title[Dimension growth and Gelfand-Kirillov dimension]{Dimension growth and Gelfand-Kirillov dimension of representations of quantum groups} 
	\author{Vyacheslav Futorny, Xingpeng Liu}
	\address[V. Futorny]{Shenzhen International Center for Mathematics, Southern University of Science and Technology, Shenzhen, Guangdong, P. R. China}
	\email{vfutorny@gmail.com}
	\address[X. Liu]{Shenzhen International Center for Mathematics, Southern University of Science and Technology, Shenzhen, Guangdong, P. R. China}
	\email{xpliu127@ustc.edu.cn}
	\subjclass[2020]{Primary: 17B37, 17B10; Secondary: 20G42, 16T20}
	\keywords{Quantum group, representation, dimension growth, Gelfand-Kirillov dimension}
	
	\begin{abstract}
		We consider two algebraic invariants in the representation theory of quantized enveloping algebras: the dimension growth of simple modules for the De Concini-Kac quantum group at roots of unity, and the Gelfand-Kirillov dimension of simple highest weight modules for the quantum group at generic~$q$. In spite of being defined for different values of the parameter~$q$, these invariants  reflect closely related features in the respective contexts. We show that several new phenomena appear in the quantum case and the
        representations with non-integral weights contribute to both invariants in a way that cannot be ignored. Building on this, we determine the minimal non-zero value of these invariants for each Lie type. As an application we show that quantum cuspidal modules at generic~$q$ can occur only when the underlying semisimple Lie algebra has simple components of type~$A$, $B$, or $C$, providing a more explicit representation-theoretic distinction with the classical case.
	\end{abstract}

	\maketitle
	\setcounter{tocdepth}{1}
	\tableofcontents
	
	\section{Introduction}
	
	Let $\rmG$ be a connected semisimple algebraic group over $\bbC$, and let $\mathfrak g$ be its Lie algebra. The present paper is concerned with the representation theory of two closely related quantum analogues of the classical enveloping algebras: the \emph{Drinfeld–Jimbo quantum group} $U_q := U_q(\mathfrak g)$, where $q$ is either a formal parameter or non-root of unity (the generic case), and the \emph{De Concini-Kac specialization} $U_\zeta$ at roots of unity. These may be viewed, respectively, as quantum counterparts of the universal enveloping algebra $U(\mathfrak g)$ in characteristic $0$ and of the universal enveloping algebra $U(\mathfrak g_k)$ for a restricted Lie algebra $\mathfrak g_k:=\frakg_\bbZ\otimes_\bbZ k $ over an algebraically closed field $k$ of characteristic $p$. It is known that the representation-theoretic phenomena in the classical settings of characteristics $0$ and $p$ mirror, to a large extent, those in the quantum settings of generic $q$ and roots of unity, respectively. The aim of this paper is to use two parallel algebraic invariants, the \emph{Gelfand–Kirillov dimensions} of simple $U_q$-modules and the \emph{dimension growth} of simple $U_\zeta$-modules (see Section \ref{DimensionGrowth}),  which play analogous roles in their respective contexts, to explain certain discrepancies between the quantum and classical pictures. 
	
	Let $\rmT$ be a maximal torus of $\rmG$, and let $U_q^0$ (resp. $U_\zeta^0$) be the Cartan part of $U_q$ (resp. $U_\zeta$). Both algebras are Laurent polynomial algebras over $\bbC$ or $\bbC(q)$ with $q$ a formal parameter, naturally isomorphic to the coordinate algebra of $\rmT$ (equivalently, to the group algebra of its character group $X(\rmT)$),  which is independent  of the value of $q$. Let $Z_{\zeta,0}$ be the Frobenius center of $U_\zeta$. It is a central subalgebra of $U_\zeta$. By \cite{DKP92}, the subalgebra $Z_{\zeta,0}$, as a Hopf subalgebra of $U_\zeta$, can be identified with the coordinate algebra of the algebraic group 
	\[\rmH = \left\{(a^-b, b\inv a^+)\in \rmB^-\times \rmB^+ \bigm| a^{\pm}\in \rmU^{\pm}, b\in \rmT \right\}, \]
	where $\rmB^+$ and $\rmB^-$ are opposite Borel subgroups of $\rmG$ satisfying  $\rmB^+\cap \rmB^- = \rmT$, a fixed maximal torus of $\rmG$, and $\rmU^{\pm}$ denote the unipotent radicals of $\rmB^{\pm}$. As a result, $\rmH$ is independent 
     of the choice of the root of unity $\zeta$, and any element of $\rmH$  defines a central character of $Z_{\zeta,0}$ for any fixed $\zeta$. 
	Let $g\in \rmH$, and denote by $\chi_g$ the corresponding central character of $Z_{\zeta,0}$. Define $U_\zeta^{\chi_g}$ (or simply $U_\zeta^{g}$) to be the quotient algebra of $U_\zeta$ by the ideal generated by $z-\chi_g(z)$ for $z\in Z_{\zeta,0}$. Define $\pi:\rmH \rightarrow \rmG$ by $\pi(x,y) = xy\inv$. By \cite{DKP92}, the algebras $U_\zeta^g$ and $U_\zeta^{g'}$ are isomorphic if $\pi(g)$ and $\pi(g')$ belong to the same conjugacy class of $\rmG$. Hence, for each conjugacy class in $\rmG$ we may choose a representative lying in $\pi(\rmH)$ and the representation theory of $U_\zeta^g$ then reduces to the situation where $\pi(g)$ is one of these chosen representatives.
	
	To clarify how dimension growth and the Gelfand–Kirillov dimension exhibit similar features in their respective contexts, we take as an example the central reduction algebra $U_\zeta^g$ with $\pi(g)$ lying in a semisimple conjugacy class of $\rmG$. In particular, when $g$ is the identity element $1$ in $\rmH$, the algebra $U_\zeta^1$, usually denoted by $u_\zeta$, is the \emph{small quantum group} introduced by Lusztig \cite[Section 8]{Lus90}. The above isomorphism implies that, for our purposes, it is enough to work uniformly with the algebra
	$\hat{u}_\zeta = u_\zeta^-\; U_\zeta^0\; u_\zeta^+ $, where $u_\zeta^\pm$ denote the positive and negative parts of $u_\zeta$, respectively. 
	Regard $\rmT$ as the \emph{maximal spectrum} $\Spec( U_\zeta^0)$ of $U_\zeta^0$. An element $\Lambda$ in $\rmT$ defines an irreducible representation of $\hat{u}_\zeta = u^-_\zeta\; U_\zeta^0\;u_\zeta^+$, namely the simple quotient of the baby Verma module of highest weight $\Lambda$, which we denoted by $L_\zeta(\Lambda)$. Let $\ell$ be the order of the root of unity and $\zeta = e^{2\pi \imath/\ell}$. We will analyze how the dimension of the irreducible representation $L_\zeta(\Lambda)$ varies as $\ell$ changes. It turns out that $\dim L_\zeta(\Lambda)$ is a polynomial function of $\ell$, and moreover,
	\[\lim_{\ell\to\infty} \log_\ell \dim L_\zeta(\Lambda) \;=\; \mathrm{GK\text{-}dim}_{U_q} L_q(\Lambda),\] 
	where $L_q(\Lambda)$ denotes the irreducible highest weight $U_q$-module of highest weight $\Lambda$, and $\mathrm{GK\text{-}dim}_{U_q} L_q(\Lambda)$ denotes its  {Gelfand-Kirillov dimension} (Theorem \ref{mainthm}). This dimension-growth behavior is consistent with the classical situation \cite{BL18}.  The equality provides a precise correspondence between the generic and root-of-unity cases in the representation theory of quantum groups, where the degree of the dimension polynomial plays a role analogous to the Gelfand–Kirillov dimension in the generic setting. Nevertheless, we  further clarify the differences from the classical case by carrying out explicit computations of the dimension growth.
	
	Let $q$ be a formal parameter. We now turn to representations of $U_q$ that are \emph{weight-graded} by $\rmT_{\bbC(q)} (\cong \Spec U_q^0)$ via the action of $U_q^0$. 
	Denote by $\frakh_\bbQ^*$ the $\bbQ$-span of the fundamental weights of $\frakg$. For each $\lambda\in \frakh_\bbQ^*$ one may define $q^\lambda$ as a weight in $\rmT_{\bfk}$ where $\bfk$ is the algebraic closure of $\bbC(q)$. Following Joseph and Letzter \cite{JL95}, we refer to such weights as \emph{linear weights}. It is well-established that the simple quotient of the Verma module with highest weight $q^\lambda$ remains simple upon the specialization $q\mapsto 1$ (see \cite{EK08} and \cite{AM12}). For these simple modules, the corresponding primitive ideals and Gelfand-Kirillov dimensions coincide with those of their classical counterpart. However, when the highest weight $\Lambda \in \rmT_{\bfk}$ is not linear, it remains unclear what the precise classical counterpart of the corresponding representation should be, and such a counterpart may in fact fail to exist.
	
	Consider the Serre subcategory of $U_q$-modules generated by all irreducible highest weight modules $L_q(\Lambda)$ with $\Lambda \in \rmT_{\bbC(q)}$ (or more general $\Lambda\in \rmT_{\bfk}$), denoted by $\mathcal{O}_q$. It is a reasonable quantum analogue of the classical Bernstein-Gelfand-Gelfand (BGG) category $\mathcal{O}(\frakg)$ associated with the Lie algebra $\mathfrak{g}$,  that contains all irreducible highest weight modules. This category shares many structural features with $\mathcal{O}(\frakg)$. In the classical case, Soergel \cite{Soe90} proved that for any weight $\lambda$ in the dual space $\frakh^*$ of the Cartan subalgebra $\frakh$ of $\frakg$, the indecomposable block $\catO_{\overline{\lambda}}(\frakg)$ of $\catO(\frakg)$ depends only on the ``integral'' Coxeter group $(W_\lambda, S_\lambda)$ relative to the weight $\lambda\in \frakh^*$ (up to group isomorphism) and on the action of $W_\lambda$ on $\lambda$, where $\overline{\lambda}$ denotes the linkage class of $\lambda$. A natural question is whether an analogous result holds for $\mathcal{O}_q$. The main issue in addressing this question is to identify appropriate integral root subsystems and the corresponding Weyl groups that govern the block decomposition of $\mathcal{O}_q$ (Subsection \ref{IntegralWeyl&subsys}). Existing results suggest that the situation should behave similarly to the classical case, and no essential differences seem to arise. However, we observe that the assignment from indecomposable blocks to their associated integral root subsystems exhibits a distinction that does not occur in the classical setting, as illustrated in the diagram below: 
	{\small \[ \Bigl\{\catO_{q,\overline{\Lambda}}\bigm\vert \Lambda\in \rmT_{\bbC(q)}\Bigr \} \rightarrow\left\{ \begin{array}{c} \text{All root} \\ \text{subsystems } \end{array} \right\} \supset \left\{ \begin{array}{c} \text{All dual-closed } \\ \text{root subsystems } \end{array} \right\}  \leftarrow  \Bigl\{\catO_{\overline{\lambda}}(\frakg)\bigm\vert \lambda\in \frakh^* \Bigr\}. \]
	}
	\vspace{-.3cm}
	
	\noindent In particular, the integral root subsystem associated with an indecomposable block of $\mathcal{O}_q$ need not be dual-closed (see Lemma \ref{dualsys} and Example \ref{Notclosed}), in contrast to the classical case. We are not aware of this result appearing in the existing literature. Therefore, we include a proof based on Jantzen's deformation theory, combined with the structure functor and the translation functor. The main idea of the proof traces back to \cite{Soe90} in the classical case and to \cite{AJS94} in the root of unity case.

	Let $\Phi$ denote the root system with respect to  $(\mathfrak{g}, \frakh)$. We show that for every proper maximal root subsystem $\Psi \subsetneq \Phi$, there exists a weight $\Lambda \in \rmT_{\bfk}$ such that the integral root subsystem associated with $\mathcal{O}_{q,\overline{\Lambda}}$ coincides exactly with $\Psi$ (Proposition \ref{maxroot-characterizations}). Based on the previous statements, we express the Gelfand-Kirillov dimension of $L_q(\Lambda)$ in terms of Lusztig’s $\afun$-function (Theorem \ref{GKdimension}), which leads to a determination of the nonzero minimal Gelfand-Kirillov dimension among all simple objects in $\mathcal{O}_q$. When $\Phi$ is indecomposable, we further obtain that this minimal nonzero value agrees with the classical one, namely, $h^\vee - 1$ (with $h^\vee$ the \emph{dual Coxeter number}), except in type $B$; whereas in type $B$, the value equals the rank of $\Phi$, say $\mathrm{rank}(\Phi) = n$, which differs from the classical value $2n-2$ (Theorem \ref{q-minGK}).	
	As an application, we can address the existence problem for quantum cuspidal weight modules at generic $q$, which is central to the classification of simple weight modules of quantum groups via parabolic induction. It also accounts for a discrepancy between quantum and classical cases noted in the previous work \cite{CGLW21}.

	The paper is organized as follows. 
	Section \ref{Preli&Nota} collects the required notations and background on quantum groups and their representation theory. In Section \ref{DimensionGrowth}, we study the growth of the dimension of irreducible $U_\zeta$-modules as the order $\ell$ of the root of unity $\zeta$ tends to infinity, and relate this behavior to the Gelfand-Kirillov dimension in the generic case. 
	Section \ref{Sec:Equiv} is devoted to the structure of the category $\catO_q$ and its block decomposition under our general weight setting; in particular, we revisit the quantum analogue of Soergel’s result. In Section \ref{minGK}, we determine the Gelfand-Kirillov dimension of simple modules in $\catO_{q}$ and compute the minimal nonzero Gelfand-Kirillov dimension. Finally, Section \ref{App} discusses an application. 
	
	\vspace{.3cm}
	{\bf Acknowledgements.} The authors  would like to thank Zongzhu Lin, Olivier Mathieu, Evgeny Mukhin, Victor Ostrik and Peng Shan for useful discussions. The second-named author is grateful to Shenzhen International Center for Mathematics for financial support and hospitality.
	V. F. was partially supported by the NSF of China (12350710787 and 12350710178); X. L. was partially supported by the China Postdoctoral Science Foundation
	(2024M751285) and the NSF of China (12401030).

	\section{Quantum groups and representations}\label{Preli&Nota}
	Throughout this paper, $\bbN$ and $\bbN_+$ denote the set of nonnegative integers and the set of positive integers, respectively. Let $q$ be an indeterminate and let $\bbC(q)$ be the field of rational functions in $q$ with complex coefficients. Set $\bfk$ to be the algebraic closure of $\bbC(q)$. For $m, r \in \bbN_+$, $m\geqslant r$, we use the following standard notations:
	\[ [m]_q = \frac{q^m-q^{-m}}{q-q\inv}, \quad [m]_q^!=[1]_q[2]_q\cdots [m]_q. \]
	We assume throughout that $\zeta \in \mathbb{C}^\times$ is a primitive $\ell$-th root of unity, where $\ell$ is an odd integer greater than $1$ and prime to $3$ if $\mathfrak{g}$ has a component of type $G_2$.
	
	\subsection{Quantum groups} Let $\frakg$ be a finite dimensional complex semisimple Lie algebra of rank $n$ with a fixed Cartan subalgebra $\frakh$. Set $I= \{1, 2, \cdots, n \}$ and let $(d_ia_{ij})_{i,j\in I}$, where $d_i$'s are relatively prime positive integers, be an $n\times n$ symmetrized  Cartan matrix of $\frakg$.  Let $W$ be the Weyl group of $\frakg$, and $\Phi$ denote the root system of $\frakg$ with respect to $\frakh$. Let $\Pi = \{\alpha_i\;|\; i \in I \}\subset \Phi$ (resp. $\{\omega_i\;|\; i\in I \} \subset \frakh^*$) be the set of \emph{simple roots} (resp. of \emph{fundamental weights}) of $\frakg$. As usual, $Q$ (resp. $P$) denotes the root (resp. weight) lattice of $\frakg$. Let $P^+ = \sum_{i\in I}\bbN \omega_i$,  $Q^+ = \sum_{i\in I}\bbN \alpha_i$ and denote  $\Phi^+= Q^+\cap \Phi$ the set of  \emph{positive} roots. Set $\rho\in P^+$ the sum of fundamental weights. There exists a non-degenerate  symmetric $W$-invariant bilinear form $(\cdot,\cdot)$ on $\frakh^*$ satisfying that $(\alpha_i, \alpha_j) = d_ia_{ij}$, $(\alpha_i, \omega_j) = d_i\delta_{ij}$ for all $i,j\in I$, and for each $ \alpha\in \Phi$ the \emph{reflection} $s_\alpha\in W$ acts on $\frakh^*$ by $s_\alpha(\lambda) = \lambda - \inprod{\lambda}{\alpha}\alpha$ for any $\lambda \in \frakh^*$, where $\alpha^\vee = \frac{2\alpha}{(\alpha, \alpha)}$.

	The Hopf algebra $U_q := U_q(\frakg)$ is generated over $\bbC(q)$ by the elements $E_i, F_i$, $ i \in I$, and $K_\alpha$, $\alpha\in Q$ subject to 
	\begin{align}
		\label{defre1}    &K_\mu K_{-\mu}=1, \quad K_\mu K_\nu =K_\nu K_\mu  \\
		\label{qc}&K_\mu E_i  = q^{(\mu,\alpha_i)}E_i K_\mu, \quad
		K_\mu F_i  = q^{-(\mu,\alpha_i)}F_iK_\mu,   \\
		& E_iF_j-F_jE_i=\delta_{ij}\frac{\Ki-\Ki\inv}{q_i - q_i\inv}, \quad q_i =q^{d_i},\label{quanfrac}
	\end{align}
	along with the quantum Serre relations (cf. \cite[Section 4.3]{Jan95}). Its comultiplication $\Delta$, counit $\epsilon$, and antipode $S$ are given by 
	\begin{equation*}
		\begin{gathered}
			\Delta(K_\mu) = K_\mu\otimes K_\mu,\quad \Delta(E_i) = E_i\otimes 1 + \Ki \otimes E_i, \quad \Delta(F_i) = F_i\otimes \Ki\inv + 1\otimes F_i, \\
			\epsilon(E_i)=\epsilon(F_i) = 0, \quad \epsilon(K_\mu) = 1, \\
			S(K_\mu) = K_{-\mu},\quad S(E_i) = -\Ki \inv E_i, \quad S(F_i) = -F_i\Ki.
		\end{gathered}
	\end{equation*}
	
	In the next subsection, we review the deformed algebra $U_R$ of $U_q$ over any deformation ring $R$, as well as the specialization of $U_q$ at values of $q $ in $ \mathbb{C}^\times$.  It is straightforward to verify that they also admit a Hopf algebra structure under the same definition. When no confusion arises, we will continue to denote the comultiplication, counit, and antipode in the corresponding algebras by $\Delta$, $\epsilon$, and $S$, respectively.

	\subsection{Specializations} \label{Specializations}
	Let $A= \bbC[q, q\inv]$. Inside $U_q$ consider the De Concini-Kac form \cite{DCK89}, that is, the $A$-subalgebra $U_A$ generated by the elements $E_i, F_i, K_{\alpha_i}$ and $\frac{\Ki-\Ki\inv}{q_i-q_i\inv}$ for $i\in I$. Let $U_A^+$ (resp. $U_A^-$) be the $A$-subalgebra of $U_A$ generated by $E_i$ (resp. $F_i$) and $U_A^0$ that generated by the $K_{\alpha_i}$ and $\frac{\Ki-\Ki\inv}{q_i-q_i\inv}$. 
	Let $R$ be an $A$-algebra. Define $U_R = U_A\otimes_A R$ with the $R$-subalgebras $U_R^\pm = U_A^\pm \otimes_A R$ and $U_R^0 = U_A^0\otimes_A R$. In particular, when $R$ is the fraction field $\bbC(q)$ of $A$ with the canonical embedding $A\hookrightarrow \bbC(q)$, we have $U_R = U_q$; in this case, we shall write the subalgebras $U_R^\pm$ (resp. $U_R^0$) as $U_q^\pm$ (resp. $U_q^0$).
	For $\xi\in \bbC^\times$, set $U_{\xi} = U_A\otimes_{A}\bbC$, where $\bbC$ is viewed as an $A$-algebra by the specialization $q \rightarrow \xi$. 
	Denote by $U_\xi^+$, $U_\xi^-$ and $U_\xi^0$ the images of $U_A^+$, $U_A^-$ and $U_A^0$ in $U_\xi$. In particular, if $\xi = \zeta$, we get the De Concini-Kac quantum group at root of unity $\zeta$, and its corresponding subalgebras $U_\zeta^\pm$ and $U_\zeta^0$. For later use, we will also consider the integral form $U_{\bbQ[q,q^{-1}]}$, obtained by replacing $A$ with $\bbQ[q,q^{-1}]$ (and similarly define $U_{\bbQ[q,q^{-1}]}^{\pm}$ and $U_{\bbQ[q,q^{-1}]}^{0}$).
If $R$ is a $\bbQ[q,q^{-1}]$-algebra, we define $U_R$, $U_R^{\pm}$, and $U_R^{0}$ by base change in the same way.

	One can introduce the quantum analog of the Poincar{\'e}-Birkhoff-Witt (PBW) basis for $U_q^\pm$. 
	Due to Lusztig \cite{Lus90} we have an action of the braid group with generators $T_i, i \in I$ by automorphisms of $U_q$ defined as follows:
	\begin{align*}
		&T_iK_\mu = K_{s_i\mu}, \quad T_iE_i = -F_i\Ki,\quad T_iF_i = -\Ki\inv E_i, \\
		&T_iE_j = \sum_{s=0}^{-a_{ij}}(-1)^{s-a_{ij}}q_i^{-s}E_i^{(-a_{ij}-s)}E_jE_i^{(s)}, \\
		&T_iF_j = \sum_{s=0}^{-a_{ij}}(-1)^{s-a_{ij}}q_i^{-s}F_i^{(s)}F_jF_i^{(-a_{ij}-s)}
	\end{align*}
	for $i\neq j$. Here we write $E_i^{(r)}$ and $F_i^{(r)}$ for the divided powers $\frac{E_i^r}{[r]_{q_i}^!}$ and $\frac{F_i^r}{[r]_{q_i}^!}$, respectively.
	
	Fix a reduced expression $s_{i_1}s_{i_2}\cdots s_{i_N}$ of the longest element $w_0$ of $W$. We then define a \emph{convex ordering} of all positive roots in $\Phi^+$:
	\[\beta_1 = \alpha_{i_1}, \quad \beta_2=s_{i_1}\alpha_{i_2}, \quad \cdots, \quad \beta_N = s_{i_1}\cdots s_{i_{N-1}}\alpha_{i_N}. \]
	Introduce root vectors $E_{\beta_r} = T_{i_1}\cdots T_{i_{r-1}}E_{i_r}$ and $F_{\beta_r} = T_{i_1}\cdots T_{i_{r-1}}F_{i_r}$. Then we choose the PBW bases as follows. For $U_q^+$ (or $U_\zeta^+$) this is a basis of monomials 
	\[E^{\underline{m}} = E_{\beta_N}^{m_N}E_{\beta_{N-1}}^{m_{N-1}}\cdots E_{\beta_1}^{m_1} \]
	where $m_i\geqslant 0$. For $U_q^-$ (or $U_\zeta^-$) this is a basis of monomials 
	\[ F^{\underline{m}} = F_{\beta_1}^{m_1}F_{\beta_2}^{m_2}\cdots F_{\beta_N}^{m_N} \]
	where $m_i\geqslant 0$. Moreover, the algebra multiplications define a $\bbC(q)$-vector space isomorphism $U_q \cong U_q^-\otimes U_q^0\otimes U_q^+$, and a $\bbC$-vector space isomorphism $U_\zeta \cong U_\zeta^-\otimes U_\zeta^0\otimes U_\zeta^+$. Note that all elements $E_\alpha^\ell$, $F_\alpha^\ell$, $\alpha\in \Phi^+$ and $K_{\alpha_i}^\ell$, $i\in I$ are central elements in $U_\zeta$. 
	
	\subsection{Centers}\label{centers}
	To start with, let $\rmG$ be the (adjoint type) complex connected algebraic group whose Lie algebra is $\frakg$, and let $\rmT$ be a fixed maximal torus of $\rmG$. Fix an $\ell$-th primitive root $\zeta$ of unity. Let $Z_{\zeta, 0}$ be the subalgebra of $U_\zeta$ generated by the elements $E_\alpha^\ell$, $F_\alpha^\ell$ for $\alpha\in \Phi^+$, and the $K_{\alpha_i}^{\pm \ell}$ for $i\in I$. Using the defining relations of $U_\zeta$, one can verify that $Z_{\zeta, 0}$ is indeed a central subalgebra of $U_\zeta$, which is independent of the choice of the reduced expression of $w_0$. 
	
	 A geometric description of $Z_{\zeta,0}$ was established in \cite{DKP92}. Let $\rmB^+$ and $\rmB^-$ be the opposite Borel subgroups of $\rmG$ with unipotent radicals $\rmU^+$ and $\rmU^-$, respectively.  Define the subgroup $\rmH\subset \rmB^-\times \rmB^+$ as the kernel of the composed morphism
	\[
	\xymatrix{
		\rmB^- \times \rmB^+ \ar[r] & \rmT \times \rmT \ar[r] & \rmT
	}
	\]
	where the first map is the natural quotient modulo the unipotent radicals, and the second is the group multiplication in $\rmT$. Notably, $\rmH$ is independent of the choice of  $\ell$ and $\zeta$,  and any element of $\rmH$ defines a central character of $\zza$. In the sequel, we treat elements of $\rmH$ (also its scalar extensions, e.g., $\rmH\times_\bbC\bbC(q)$) as ``\emph{universal}'' characters of $\zza$, which can be specialized at any given choice of $\ell$ and $\zeta$ as needed. 
	Let $g\in \rmH$, and denote by $\chi_g$ the corresponding central character of $Z_{\zeta,0}$. Define $U_\zeta^{\chi_g}$ (or simply $U_\zeta^{g}$) as the quotient algebra of $U_\zeta$ by the ideal generated by $z-\chi_g(z)$ for $z\in Z_{\zeta,0}$. Define $\pi:\rmH \rightarrow \rmG$ by $\pi(x,y) = xy\inv$. By \cite{DKP92}, the algebras $U_\zeta^g$ and $U_\zeta^{g'}$ are isomorphic if $\pi(g)$ and $\pi(g')$ belong to the same conjugacy class of $\rmG$. 
	
	Note that there exists a canonical  $W$-action on $U_q^0$ given by $wK_\mu = K_{w\mu}$ for $w\in W$ and $\mu\in Q$. 
	Define an action of $\bbZ_2^n$ on $U_q^0$ by 
    \begin{equation}\label{Z2action}
        \sigma(K_{\mu}) = \sigma(\mu)K_{\mu}, \quad \text{for } \sigma \in \bbZ_2^n, \mu\in Q,
    \end{equation}
	where $ \sigma(\mu)\in \{\pm 1 \}$ is determined by extending the action on Cartan generators that the $i$-th copy of $\bbZ_2^n$ multiply $K_{\alpha_i}$ by $1$ or $-1$. Combined with the above, we obtain an action of the \emph{extended Weyl group} $\Wch=\bbZ_2^n\rtimes W$ on $U_q^0$. 
	
	Let $\rmT_q = \Hom_{\bbC(q)\text{-}\mathrm{alg}}\big(U_q^0, \bbC(q)\big)$. Then the above action induces a $\Wch$-action on $\rmT_q$ given as 
    \begin{equation}\label{Waction}
        (w\Lambda)(K_\mu) = \Lambda(w\inv K_\mu)
    \end{equation}
	for any $w\in \Wch$, $\Lambda\in \rmT_q$ and $\mu\in Q$. Note that the action of $W$ on $\rmT_q$ preserves the group structure of $\rmT_q$, i.e., $w(\Lambda\Lambda^\prime) = (w\Lambda)(w\Lambda^\prime)$ for $w\in W$ and $\Lambda, \Lambda^\prime\in \rmT_q$. 
	
	Define an algebra homomorphism $\gamma: U_q^0\rightarrow U_q^0$ by $\gamma(K_i) = q_iK_i$. One may check that  $\Lambda\circ\gamma = q^\rho \Lambda$ for $\Lambda \in \rmT_q$. Let $U^+_+$ (resp. $U^-_+$) be the \emph{augmentation ideal} of $U_q^+$ (resp. $U_q^-$), that is, the kernel of the counit $\epsilon$ restricted to $U_q^+$ (resp. $U_q^-$). Using the PBW theorem we obtain that $U_q$ has a direct sum decomposition 
	\[U_q = U_q^0 \oplus (U_+^-U_q+U_qU^+_+). \]
	Let $\pr_{U_q}$ be the projection of $U_q$ onto $U_q^0$. Now restrict the projection map $\pr_{U_q}$ to the center $Z(U_q)$ of $U_q$. Then from \cite{DCK89} there exists a $\bbC(q)$-algebra isomorphism, known as \emph{the Harish-Chandra isomorphism}, 
	\[
	\xymatrix{
		{\gamma^{-1}\circ \pr_{U_q}}: Z(U_q) \ar[r]^-{\widesim{}} &
		(U_q^0)^{\Wch}
		:= \{\,u \in U^0 \mid wu = u,\ \forall w \in \Wch\,\}.
	}
	\]
	
	Let $\chi_\Lambda$ denote the \emph{central character} of $Z(U_q)$ associated with a weight $\Lambda \in \rmT_q$, defined by $\chi_\Lambda(z)=\Lambda(\pr_{U_q}( z))$ for $z\in Z(U_q)$.  Note that  $\Lambda(\pr_{U_q} (z))$ agrees with $(\Lambda q^\rho)(\gamma\inv\circ\pr_{U_q}(z))$, from which the Harish-Chandra isomorphism yields the following result.  
	
	\begin{lemma}
		$\chi_\Lambda = \chi_{\Lambda^\prime}$ if and only if $\Lambda q^\rho $ and $\Lambda^\prime q^\rho $ are conjugate under $\Wch$. \qed 
	\end{lemma}
	
	In the classical setting, $W$ admits a \emph{shifted action} on $\frakh^*$ given by $w\cdot \lambda = w(\lambda + \rho)-\rho$ for $w\in W$, $\lambda \in \frakh^*$. Analogously, we introduce the following shifted action of $\Wch$ on $\rmT_q$. 
	\begin{definition}
		Let $w \in \Wch$ and $\Lambda \in \rmT_q$. Define 
		$w\cdot \Lambda = q^{-\rho}w(\Lambda q^\rho)$.
	\end{definition}
	
	Specializing $q$ to $\zeta$, we get a central subalgebra of $U_\zeta$ from $Z(U_q)$, namely the \emph{Harish-Chandra center}, which we denote by $Z_{\zeta,1}$. By the Harish-Chandra isomorphism, $Z_{\zeta,1}$ is isomorphic to the $\Wch$-invariant subalgebra  $(U_\zeta^0)^{\Wch}$ of $U_\zeta^0$. Consider the shifted action of $\Wch$ on $\rmT$ (similar as the action on $\rmT_q$).  Take  $[\Lambda] \in \rmT/\big(\Wch, \cdot\big)$. The central reduction $U_{\zeta,\Lambda}$ of $U_\zeta$ is then defined as the quotient of $U_\zeta$ by the ideal generated by $z-\chi_\Lambda(z)$, $z\in Z_{\zeta,1}$. 
	Let $\rmT/\Wch$ denote the quotient of $\rmT$ by the ordinary $W$-action and sign translation by $\bbZ_2^n$. Consider the fiber product 
	$\rmH \times_{\rmT/\Wch} \rmT/(\Wch, \cdot) $, where the map $\rmT/(\Wch, \cdot) \rightarrow \rmT/\Wch$ sends $[\Lambda] \mapsto [\Lambda^\ell]$, and the map $\rmH \rightarrow \rmT/\Wch$ is the composition of $\pi:\rmH \rightarrow \rmG$ and \emph{the Steinberg map} $\rmG \rightarrow \rmT/\Wch$. Here a potential confusion may arise: we regard $\rmT$ as a subgroup of $\rmH$ in a natural way. The map $\pi$ then induces the square map $\rmT \to \rmG$, whose image $\rmT/\mathbb{Z}_2^n := \overline{\rmT}$ is exactly the maximal torus of $\rmG$. Thus, we can consider the Steinberg map $\rmG \to \overline{\rmT}/W (\cong \rmT/\Wch)$, see \cite[Subsections 16.2,16.3]{DCP93} for more details. Note that every irreducible $U_\zeta$-module factors uniquely through a central reduction $U_{\zeta,\Lambda}^{g}$ for some $(g, [\Lambda])\in\rmH \times_{\rmT/\Wch} \rmT/(\Wch, \cdot) $.

	\subsection{Representations} 	Let $L$ be an irreducible $U_\zeta$-module. Note that $Z_{\zeta,0}$ acts on $L$ by a central character in $\Spec Z_{\zeta,0}$. In particular, the central elements $K_\mu^\ell$, $\mu\in Q$ act by scalars. It follows that $L$ is a weight  $U_\zeta$-module in terms of $U_\zeta^0$ with its weights lying in $ \rmT$. 
	Let $R$ be an $A$-algebra (or a $\bbQ[q,q\inv]$-algebra). We can consider more general \emph{weights} for $U_R$. Define 
	\[\rmT_R:=\Hom_{R\text{-}\mathrm{alg}}\big(U_R^0, R\big),\]
	the set of all $R$-algebra homomorphisms from $U_R^0$ to $R$.
	Thanks to the Hopf algebra structure of $U_\zeta^0$ (inherited from $U_\zeta$),  $\rmT_R$ carries a natural abelian group structure, where the multiplication and the inverse are given by 
	\[(\Lambda \Lambda^\prime)(u) = (\Lambda\otimes \Lambda^\prime)\circ \Delta(u), \quad \Lambda\inv(u) = \Lambda\circ S(u) \]
	for any $\Lambda, \Lambda^\prime\in \rmT_R$, and $u \in U_R^0$. Moreover, an $A$-algebra homomorphism $R \rightarrow R^\prime$ induces a group  homomorphism $\rmT_R \rightarrow \rmT_{R^\prime}$; in particular, the inclusions $\bbQ[q,q\inv]\subset A \subset \bbC(q) \subset \bfk$ yields that 
	\[ \rmT_{\bbQ[q,q\inv]} \subset \rmT_A \subset \rmT_{q} \subset \rmT_\bfk. \]
	Now identify $\frakh_{\bbQ}^*$ with the $\bbQ$-linear span of the fundamental weights $\{\omega_i \;|\; i\in I  \}$. Each $\lambda \in \frakh_{\bbQ}^*$ gives rise to a weight $q^\lambda$ in  $\rmT_\bfk$ by assigning $ K_\mu $ to $ q^{(\lambda, \mu)} $. We refer to such a weight $q^\lambda$, for $\lambda\in \frakh_{\bbQ}^*$, as a \emph{linear weight} in $\rmT_\bfk$ following Joseph-Letzter \cite{JL95}. Note that  $q^P:=\{q^\lambda\;|\; \lambda \in P\}$ lies in $ \rmT_A$. Specializing $q$ to $\zeta$, the linear weight $q^\lambda$ induces a weight $\zeta^\lambda \in \rmT$, defined by $K_\mu \mapsto \zeta^{(\lambda, \mu)}$. 
	
	Let $M$ be a $U_R$-module and $\Lambda \in \rmT_R$.  Define 
	\[M_\Lambda = \{v \in M \; |\;  u.v = \Lambda(u)v, \forall u \in U_R^0\}.\]
	Then $M_\Lambda$ is an $R$-module. By the defining relations (\ref{qc}) we have $ E_i.M_\Lambda \subset M_{\Lambda q^{  \alpha_i}} $ and $ F_i.M_\Lambda \subset M_{\Lambda q^{-  \alpha_i}} $. If $M_\Lambda$ is nonzero, then we say that $\Lambda$ is a \emph{weight} of $M$, and $M_\Lambda$ is a \emph{weight $R$-module} of \emph{weight} $\Lambda$. Call 
	$M$  a \emph{weight module} of $U_R$ if $M = \oplus_\Lambda M_\Lambda$. In this case, we denote by $\Omega(M)$ the set of all weights of $M$. 
	In general, the notion of formal characters for weight modules over $U_R$ can be defined as follows (cf. \cite[Section 4.1]{Jos95}). 
	\begin{definition}
		Let $M = \oplus_{\Lambda\in \rmT_R}M_\Lambda$ be a weight $U_R$-module such that each  $M_\Lambda$ is a free $R$-module of finite rank. Then $M$ admits a \emph{formal character}, defined by 
		\[ \ch M = \sum_{\Lambda\in \rmT_R} (\mathrm{rank}_{R} M_\Lambda)\Lambda. \]  
	\end{definition}
	
	For $\Lambda \in \rmT_R$, let 
	$M_R(\Lambda) = U_R\otimes_{U_R^0U_R^+}R_\Lambda$ denote the \emph{Verma module} of $U_R$, where $R_\Lambda$ is the  $U_R^0U_R^+$-module on which $U_R^+$ acts trivially and $u.1_R = \Lambda(u)$ for all $u\in U_R^0$. Obviously, $M_R(\Lambda)$ is a weight module.
	Let $\bbF$ be a subfield of $\bfk$ containing $\bbQ[q,q\inv]$, and suppose $\Lambda \in \rmT_\bbF$. We then write $M_q(\Lambda)$ and $L_q(\Lambda)$ for the Verma module $M_{\bbF}(\Lambda)$ over $U_\bbF$ and its simple quotient $L_\bbF(\Lambda)$, respectively.

	Assume $\Lambda\in \rmT_q$. Let $\ell$ be a sufficiently large positive integer such that $\Lambda \in \rmT_{A_\zeta}$, where $A_\zeta = \bbC[q]_{(q-\zeta)}\subset \bbC(q)$  is the  localization of $\bbC[q]$ at the maximal ideal $(q-\zeta)$. Since $A\subset A_\zeta$, the ring $A_\zeta$ naturally inherits an $A$-algebra structure. 
	\begin{definition}
		Let $\Lambda \in \rmT_{A_\zeta}$. Define $\Lambda_\zeta\in \rmT$ to be the image of $\Lambda$ under the \emph{specialization map} $(-)_\zeta: \rmT_{A_\zeta} \rightarrow \rmT$. 
	\end{definition}
	We may define an $A_\zeta$-form of the Verma module $M_q(\Lambda)$. Specifically, choose a highest weight vector $v_\Lambda$ and consider its $U_{A_\zeta}$-submodule  $M_{A_\zeta}(\Lambda) = U_{A_\zeta}.v_\Lambda $. 
	This defines an $A_\zeta$-submodule of $M_q(\Lambda)$, and its specialization at $q=\zeta$ yields a $U_\zeta$-module $M_\zeta(\Lambda_\zeta) := M_{A_\zeta}(\Lambda) / (q-\zeta)M_{A_\zeta}(\Lambda)$, which we also denote simply by $M_\zeta(\Lambda)$. 
	Note that the $U_\zeta$-module $M_\zeta(\Lambda)$ is indecomposable but reducible. 
	Let $\overline{M}_\zeta(\Lambda)$ denote the quotient of $M_\zeta(\Lambda)$ by the $U_\zeta$-submodule generated by all the vectors $F_\alpha^\ell.\bar{v}_\Lambda$, $\alpha\in \Phi^+$. We call $\overline{M}_\zeta(\Lambda)$ a \emph{diagonal  $U_\zeta$-module} (cf. \cite[Section 3.2]{DCK89}).

	\subsection{Jantzen's filtration}\label{Janfiltration}
	The quantum group $U_q$ admits a  \emph{quantum Shapovalov form}, introduced in \cite{DCK89,JL95}.
	Define a bilinear form $\calS$ on $U_q$ by $\calS(a, b) = \pr(\delta(a)b)$ for $a, b \in U_q$, where $\delta: U_q \rightarrow U_q$ is the $\bbC(q)$-algebra anti-automorphism defined by 
	\[\delta(E_i) = F_i, \quad \delta(F_i) = E_i, \quad  \delta(K_i)=K_i, \quad \text{for}\; i\in I.  \]
	One may check that $\calS$ is symmetric and contravariant, i.e., $\calS(ab, c) = \calS(b, \delta(a)c)$ for $a, b, c \in U_q$.
	The form $\calS$ can be computed from the restriction to $U^-_q$. For $\nu\in Q^+$, let $U_{-\nu}^-$ be the $\bbC(q)$-linear subspace of $U_q^-$ spanned by the basis $\{F^{\underline{m}}\;|\; \sum_im_i\beta_i = \nu \}$. Note that $\calS(U_{-\mu}^-, U_{-\nu}^-) = 0$ unless $\mu = \nu$ in $Q^+$. Define $\calS_\nu$ to be the restriction of $\calS$ to $U_{-\nu}^-$. Then the matrix $\calS_\nu$ with respect to the basis $\{F^{\underline{m}}\;|\; \sum_im_i\beta_i = \nu \}$ has entries in $U_q^0$. Due to the special nature of commutation relations between the $E_i, F_i$, these entries actually lie in the subalgebra of $U_q^0$ generated by the elements $\Ki$, $i\in I$, over the ring $\bbQ[q, q\inv, (q_i-q_i\inv)\inv$, $i\in I]$, namely $U^0_{\bbQ[q,q\inv]}$. 
	
	For a positive root $\alpha\in \Phi^+$,  with $\alpha = w(\alpha_i)$ for some $i\in I$, $w\in W$, set $q_\alpha = q_i$ and $[K_\alpha;m]=w\left([K_{\alpha_i};m]\right)$, where $m\in \bbZ$ and 
	\[[K_{\alpha_i}; m]_q :=\frac{K_{\alpha_i} q^m -K_{\alpha_i}\inv q^{-m}}{q_i -q_i\inv}.  \]
	Then we can formulate  
	\begin{equation}\label{Shapdet}
		\det(\calS_\nu) = \prod_{\alpha\in \Phi^+}\prod_{m\in \bbN_+}\left([m]_{q_\alpha}[K_\alpha; (\rho, \alpha)-\frac{m}{2}(\alpha, \alpha)]_q\right)^{\rmp(\nu-m\alpha)},
	\end{equation}
	where $\rmp$ is the \emph{Kostant partition function} $Q^+ \rightarrow \bbN$. 
	
	The importance of $\calS$ lies in the fact that it induces a bilinear form on each Verma module. Let $R$ be a $\bbQ[q,q\inv, (q_i-q_i\inv)\inv$, $i\in I]$-algebra. For any $\Lambda\in \rmT_R$, the map 
	\[ (av_\Lambda, bv_\Lambda) \mapsto (\Lambda\circ\calS)(a, b) \]
	defines a symmetric $R$-bilinear form on $M_R(\Lambda)$,  denoted by $\calS_R^\Lambda$. In particular, if $R$ is further a field, then it can be proved that the radical of $\calS_R^\Lambda$ is exactly the unique maximal submodule of $M_R(\Lambda)$. Hence $M_R(\Lambda)$ is irreducible if and only if the radical of $\calS_R^\Lambda$ is trivial. For example, when $R = \bfk$, 	a standard argument then gives 
	
	\begin{lemma}\label{Verma}
		\begin{enumerate}[{\rm (1)}]
			\item The Verma module $M_q(\Lambda) $ is irreducible if and only if $q^{2(\rho, \alpha)}\Lambda(K_{2\alpha}) \notin q^{\bbN_+(\alpha, \alpha)}$ for any $\alpha\in \Phi^+$.
			
			\item If $q^{2(\rho, \alpha)}\Lambda(K_{2\alpha}) = q^{m(\alpha,\alpha)}$ for some $m\in \bbN_+$, then $M_q(\Lambda q^{-m\alpha})$ can be embedded into $M_q(\Lambda)$. \qed 
		\end{enumerate}
	\end{lemma}  

	Recall that the Jantzen filtration for the universal enveloping algebra $U(\frakg)$ offers a conceptual approach to the structure of  Verma modules. Its $q$-analogue plays a similar role for quantum groups. 
	
	Let $\bbF$ be a subfield of $\bfk$ containing  $\bbQ[q,q\inv, (q_i-q_i\inv)\inv$, $i\in I]$. Set $B=\bbF[t, t\inv]$, where $t$ is an indeterminant. Let $T=(t-1)$ denote the maximal ideal of $B$, equipped with the valuation 
	\[ \mathrm{val}_T: f(t)\mapsto m \quad \text{ if\;} f(t)\in T^m\setminus T^{m+1} \text{ for\;} m\in \bbN.   \]
	Let $\Lambda$ be a weight in $ \rmT_\bbF$. Define $\Lambda^\prime = \Lambda t^\rho\in \rmT_B$, where $t^\rho:K_\mu \mapsto t^{(\rho,\mu)}\in \rmT_B$.  Consider the one-parameter deformation $M_B(\Lambda^\prime)$. For each $i\in \bbN$ define 
	\[F^iM_B(\Lambda^\prime) = \left\{v\in M_B(\Lambda^\prime) \bigm| \calS^{\Lambda^\prime}\left(v, M_B(\Lambda^\prime)\right)\subset T^i \right\}. \]
	By contravariance of $\calS^{\Lambda^\prime}$, the $F^iM_B(\Lambda^\prime)$ define a decreasing sequence of $U_B$-submodules $M_B(\Lambda^\prime)=F^0M_B(\Lambda^\prime)\supset F^1M_B(\Lambda^\prime)\supset F^2M_B(\Lambda^\prime)\supset \cdots$ with $F^iM_B(\Lambda^\prime)=0$ for large enough $i$. Set $F^iM_q(\Lambda) = F^iM_B(\Lambda^\prime)\otimes_BB/T$. Then the $F^iM_q(\Lambda)$ form a decreasing sequence of $U_\bbF$-submodules: 
	\[ M_q(\Lambda)=F^0M_q(\Lambda)\supset F^1M_q(\Lambda)\supset F^2M_q(\Lambda)\supset \cdots \]
	with $F^iM_q(\Lambda)=0$ for large enough $i$. This filtration is called \emph{the Jantzen filtration} of $M_q(\Lambda)$. Here $F^1M_q(\Lambda)$ is the unique maximal submodule of $M_q(\Lambda)$,  each nonzero quotient $F^iM_q(\Lambda)/F^{i+1}M_q(\Lambda)$ has a non-degenerate contravariant form, and for each $\beta\in Q^+$ we have $\dim_\bbF F^iM_q(\Lambda)_{\Lambda q^{-\beta}}  =\mathrm{rank}_{B} F^iM_{B}(\Lambda^\prime)_{\Lambda^\prime q^{-\beta}}$, where $F^iM_{B}(\Lambda^\prime)_{\Lambda^\prime q^{-\beta}} $ is a free $B$-module. 
	
	The character sum formula for the submodule filtration $\{F^iM_q(\Lambda) \}_i$ can be determined. Let us first introduce the following notation.
	\begin{definition}\label{Subset}
		Define $\calT_{\Lambda}$ a subset of $ \bbN_+\times \Phi^+$ as follows:
		\[\calT_{\Lambda} =  \big\{(m,\alpha)\in \bbN_+\times \Phi^+\bigm| q^{2(\rho,\alpha)-m(\alpha,\alpha)}\Lambda(K_{2\alpha}) =1 \big\}. \]
	\end{definition}
	
	Fix $\nu\in Q^+$. Consider the determinant (\ref{Shapdet}) specialized at the $\Lambda^\prime q^{-\nu}$-weight space of $M_B(\Lambda^\prime)$. The value of $\mathrm{val}_T$ at the factor $\Lambda^\prime\left([K_\alpha;(\rho,\alpha)-\frac{m}{2}(\alpha,\alpha)]\right)$ is $0$ unless $(m,\alpha)\in \calT_{\Lambda}$, in which case the value is $1$. As a result, the contribution to the character sum $\sum \ch F^iM_q(\Lambda)$ for fixed $\nu$ and $(m, \alpha)\in \calT_{\Lambda}$ is $\rmp(\nu-m\alpha)\Lambda q^{-\nu}$. 
	Therefore, we obtain the following  \emph{Jantzen sum formula}
	\begin{equation}\label{Jantzenformula}
		\sum_{i=1}^\infty \ch F^iM_q(\Lambda) =\sum_{\nu\in Q^+}\sum_{(m,\alpha)\in \calT_{\Lambda}}\rmp(\nu-m\alpha)\Lambda q^{-\nu} = \sum_{(m,\alpha)\in \calT_{\Lambda}}\ch M_q(\Lambda q^{-m\alpha}). 
	\end{equation}
	
	\begin{remark}
		For any root of unity $\zeta$, observe that $\bbQ[q, q\inv, (q_i-q_i\inv)\inv$, $i\in I]$ is contained in $ A_\zeta$. Let $\Lambda\in \rmT_{A_\zeta}$. Applying the product formula \eqref{Shapdet}, with the factors indexed by $1 \le m \le \ell-1$, to the diagonal $U_\zeta$-module $\overline{M}_\zeta(\Lambda)$, we likewise obtain a finite decreasing sequence of $U_\zeta$-submodules together with a corresponding sum formula for $\overline{M}_\zeta(\Lambda)$. 
	\end{remark}

	\section{Dimension growth and Gelfand-Kirillov dimension}\label{DimensionGrowth}
	In this section, we study the growth of the dimension of irreducible representations of  $U_\zeta$, as the order $\ell$ of the root of unity $\zeta$ tends to infinity (with $\zeta$ varying accordingly), and establish a connection with the Gelfand-Kirillov dimension of irreducible representations in the generic case.
	
	\subsection{Polynomial growth} \label{Polynomialgrowth} Let $f: \bbN \rightarrow \bbN$ be a function.  The \emph{polynomial growth} of $f$ is a number in $\bbR\cup \{\infty \}$ defined by 
	\[ d_f = \limsup_{m\rightarrow \infty} \frac{\log f(m)}{\log m}.  \]
	
	Let $\bfF$ be any field of characteristic $0$. Let $\calA$ be a finitely generated, unital, associated  $\bfF$-algebra and $V\subset \calA$ be a generating $\bfF$-subspace. Set $\calA_n = \sum_{i=0}^nV^i \subset \calA$ with $V^0 = \bfF.1$ and $f_\calA(n)= \dim_\bfF \calA_n$. Then the \emph{Gelfand-Kirillov (GK) dimension} of $\calA$ is defined by $d_{f_\calA}$. Similarly, for a finitely generated $\calA$-module $M$, we choose $M^0$ to be a finite dimensional $\bfF$-subspace of $M$ generating $M$ as an $\calA$-module. Set $M^n = V^nM^0$ and $f_M(n) = \dim_\bfF M^n$. Then the \emph{GK-dimension} of $M$ is defined as $d_{f_M}$. Note that $d_{f_\calA}$ and $d_{f_M}$ is independent of the choice of the generating subspaces. Therefore, we shall denote the GK-dimensions $d_{f_\calA}$ and $d_{f_M}$ by $d(\calA)$ and $d_\calA(M)$ (or simply $d(M)$), respectively. For more general theory one can refer to \cite{KL00}. 
	
	\vspace{.3cm}
	The following results can be found in  \cite[Appendix A.3.6]{Jos95}.
	
	\begin{lemma}\label{GKdimgeneralprop}
		Let $N\subset M$ be finitely generated $\calA$-modules and $V$ a finite dimensional $\calA$-module. Then 
		\begin{enumerate}[{\rm (1)}]
			\item $d_\calA( M) \geqslant \max\{d_\calA(N), d_\calA(M/N) \}$. 
			\item $d_\calA( M) = \max\{ d_\calA(\calA m_i)\;|\; i =1, \dots, r\}$, for any finite set $m_1, \dots, m_r$ of generators of $M$. 
			\item If $\calA$ is a Hopf algebra with its antipode being surjective, then 
			\[\pushQED{\qed}
			d_\calA (M\otimes V) = d_\calA (M) = d_\calA (V\otimes M).\qedhere 
			\popQED
			\]
		\end{enumerate}	
	\end{lemma}
	
	For any finite dimensional $U_q$-module $M$, it is clear that $d_{U_q}(M) = 0$. Suppose that $M$ is a highest weight $U_q$-module. Since $M$ is generated by its highest weight vector under the action of $U_q^-$,  we have $d_{U_q}(M) = d_{U_q^-}(M)$. For the Verma module $M_q(\Lambda)$, the GK-dimension is given by  $d_{U_q}(M_q(\Lambda)) = |\Phi^+|$.  Hence it follows that  $d_{U_q}(L_q(\Lambda)) \leqslant |\Phi^+|$. 
	
	\vspace{.3cm}
	A more detailed discussion of the GK-dimension formula for  $L_q(\Lambda)$ will be presented in Section \ref{minGK}. In what follows, we turn our attention to the dimension growth of irreducible representations at roots of unity.
	
	\subsection{Dimension growth and GK-dimension} 
	Denote by $u_\zeta$ the small quantum group, defined as the quotient of $U_\zeta$ by the ideal generated by the $\ell$-th powers $E_\alpha^\ell$, $F_\alpha^\ell$, $\alpha\in \Phi^+$ and $K_{\alpha_i}^\ell$, $i\in I$. It has the triangular decomposition $u_\zeta = u_\zeta^-u_\zeta^0u_\zeta^+$ with the obvious definitions of the three parts. Define 
	$\hat{u}_\zeta = u_\zeta^-U_\zeta^0u_\zeta^+. $
	Then we have 
	\begin{lemma}
		Let $L$ be a simple $\hat{u}_\zeta$-module. Then there exists a unique diagonal Verma module $\overline{M}_\zeta(\Lambda)$ for some $\Lambda\in \rmT$ such that $L$ is its homomorphic image; we denote this simple quotient by $L_\zeta(\Lambda)$.
	\end{lemma}
	\begin{proof}
		It is clear that $L$ is finite dimensional. Since the minimal polynomial of each operator $K_{\alpha_i}$ on $L$ has distinct roots, $L$ is a weight module of $\hat{u}_\zeta$. Note that all operators $E_\alpha$, $\alpha\in \Phi^+$ are nilpotent on $L$ since $E_\alpha^\ell =0$. Therefore, we can find a nonzero vector $v\in L$ (unique up to scalar) such that $E_\alpha.v=0, K_{\alpha_i}.v=a_iv, i\in I$ for some scalars $a_i\in \bbC^\times$. It follows that $L = \hat{u}_\zeta.v$ due to the simplicity of $L$.  Therefore, $L$ is the homomorphic image of $\overline{M}_\zeta(\Lambda)$ with $\Lambda\in \rmT$ determined by $K_{\alpha_i}\mapsto a_i$ for each $i\in I$.  
	\end{proof}

	For any $\Lambda\in \rmT_q$, there exists a positive integer $\ell_0$ such that $\Lambda\in \rmT_{A_\zeta}$ for all roots of unity $\zeta$ of order $\ell \geqslant \ell_0$. 
	For technical convenience, we have assumed at the beginning of the paper that $\ell$ is an odd positive integer, along with certain standard simplifying assumptions.

	Now we fix $\zeta = e^{\frac{2\pi\imath}{\ell}} $ with $\imath = \sqrt{-1}\in \bbC$. For $ \Lambda \in \rmT_q $, we say that a positive integer $ \ell $ is \emph{admissible} with respect to $ \Lambda$  if $\ell$ satisfies the above assumptions and $\Lambda \in \rmT_{A_\zeta} $.
	This defines a function 
	\[\left\{ \text{admissible}\; \ell\right\} \rightarrow \bbN, \quad \ell \mapsto \dim L_{\zeta}(\Lambda).  \]
	Then the polynomial growth of this function can be described as follows.
	
	\begin{theorem}\label{mainthm}
		For any  $\Lambda \in \rmT_q$, we have 
		\[ \lim_{\ell\rightarrow \infty} \frac{\log \dim L_{\zeta}(\Lambda)}{\log \ell} = d_{U_q} \left(L_q(\Lambda)\right), \]
		where the limit is taken over all admissible integers $\ell$ with respect to $\Lambda$.
	\end{theorem}

	The proof  of Theorem \ref{mainthm} will be carried out in the subsequent subsections following an approach adapted from Bezrukavnikov and Losev \cite{BL18} in the context of restricted Lie algebras,  which is built on a method of Etingof for rational Cherednik algebras.

	\subsection{Proof of Theorem \ref{mainthm}} 
	Let $h=h(q) \in \bbC(q)$ be a rational function. Then $h$ defines a holomorphic function on $\bbC\setminus\Sing(h)$, where $\Sing(h)\subset \bbC$ is the finite set of all \emph{poles} of $h$. Let $\Tor(\bbC^\times)$ denote the \emph{torsion subgroup} of the multiplicative group $\bbC^\times$, i.e., the set of all roots of unity in $\bbC$. 
	
	The following elementary fact from complex analysis will be useful. 
	\begin{lemma}\label{rootunity}
		Suppose that $h$ maps infinitely many elements of $\Tor(\bbC^\times)$ into $\Tor(\bbC^\times)$. Then there exist $\zeta_0\in \Tor(\bbC^\times)$ and $m\in \bbZ$ such that $h(q) = \zeta_0 q^{m}$ for all $q\in \bbC\setminus\Sing(h)$. 
	\end{lemma}
	\begin{proof}
		Note that $\overline{h(q)} = h(\bar{q}\inv)\inv$ for infinitely many $q\in \Tor(\bbC^\times)$, so by analytic continuation, the identity holds for all $q\in \bbC\setminus\Sing(h)$. This implies that  the nonzero poles and zeros of $h$ coincide, and hence $h$ must be a product of functions of the form $q^m$ and $\frac{q-a}{1-\bar{a}q}$ where $m\in \bbZ$ and $a\in \Tor(\bbC^\times)$. This completes the proof.
	\end{proof}
	
	Observe that $U_q^-$, and hence $M_q(\Lambda)$ and its specialization $M_\zeta(\Lambda)$,  carry a natural $\bbN$-grading given by $\deg F_i = 1$ for all $i \in I$. The corresponding irreducible modules $L_q(\Lambda)$ and $L_\zeta(\Lambda)$ then inherit this grading as their respective graded quotients. We write $M_j$ for the \emph{$j$-th graded component} of a graded module $M$. Then we have the following result.
	
	\begin{lemma}\label{dimensionformula}
		Let $\Lambda\in \rmT_q$, and assume that $\ell$ is a sufficiently large admissible integer. Then there exists a positive integer $m_\Lambda$, independent of $\ell$, such that 
		\[\dim_{\bbC(q)}L_q(\Lambda)_j = \dim_{\bbC}L_\zeta(\Lambda)_j \]
		for all $j< \frac{\ell }{m_\Lambda}$.
	\end{lemma}
	\begin{proof}
		Consider an $A_\zeta$-algebra $R$. We primarily focus on two specific $A_\zeta$-algebras: (1) $R=\bfk$, where the $A_\zeta$-algebra structure is given by the natural  inclusion $A_\zeta \hookrightarrow \bfk$; and (2) $R=\bbC$, where $A_\zeta \rightarrow \bbC$ via the specialization $q\mapsto \zeta$. Note that  $\bbC[q, q\inv, (q_i-q_i\inv)\inv, i\in I]\subset A_\zeta$. Then the Verma module $M_R(\Lambda^\prime)$ admits a contravariant form $\calS_R^{\Lambda^\prime}$, which is the specialization of the $R$-deformation form of $\calS$ (see Subsection \ref{Janfiltration}) at $\Lambda^\prime\in \rmT_R$. For the above two cases, we simplify the notations by  $\calS_q^{\Lambda^\prime}$ and $\calS_\zeta^{\Lambda^\prime}$, respectively. Let  $\calS_{R,j}^{\Lambda^\prime}$ denote the restriction of $\calS_R^{\Lambda^\prime}$ to the $j$-th graded component $M_R(\Lambda^\prime)_j$. Consider the one-parameter deformation $M_{R[t^{\pm1}]}(\Lambda t^\rho)$, where $t$ is an independent variable. Similarly, we write  $\calS_q^{\Lambda t^\rho}$ and $\calS_\zeta^{\Lambda t^\rho}$ for $\calS_{\bfk[t^{\pm1}]}^{\Lambda t^\rho}$ and $\calS_{\bbC[t^{\pm1}]}^{\Lambda t^\rho}$, respectively. 
		
		Notice that $L_q(\Lambda)\otimes_{\bbC(q)} \bfk$ (resp. $L_\zeta(\Lambda)$) is the quotient of $M_q(\Lambda)\otimes_{\bbC(q)} \bfk$ (resp. $M_\zeta(\Lambda)$) by the radical of $\calS_q^\Lambda$ (resp. $\calS_\zeta^\Lambda$).  Therefore, it suffices to show that there exists a positive integer $m_\Lambda$, independent of $\ell$, such that the radicals of $\calS_{q,j}^\Lambda$ and of  $\calS_{\zeta, j}^\Lambda$ have the same dimension for all  $j< \frac{\ell }{m_\Lambda}$. 
		
		We first claim that there exists a positive integer $m_\Lambda$, independent of $\ell$, such that 
		\begin{equation}\label{valuation}
			\val\left(\det \calS_{\zeta,j}^{\Lambda t^\rho}\right)=	\val\left(\det \calS_{q,j}^{\Lambda t^\rho}\right) 
		\end{equation}
		for all $j< \frac{\ell }{m_\Lambda}$. Indeed, from (\ref{Shapdet}), we know that for $j<\ell$, 
		\[\det \calS_{\zeta,j}^{\Lambda t^\rho}=\prod_{\nu,\alpha}\prod_{1\leqslant m\leqslant j}\left([m]_{\zeta_\alpha}\Lambda_\zeta t^\rho\left([K_\alpha; (\rho, \alpha)-\frac{m}{2}(\alpha, \alpha)]_\zeta\right)\right)^{\rmp(\nu-m\alpha)},\]
		where the product runs over all $\nu\in Q^+$ of \emph{height} $j$ and $\alpha\in \Phi^+$. The determinant $\det \calS_{q,j}^{\Lambda t^\rho}$ has the same form, after replacing $\zeta$ by $q$. The value of $\val$ at $\det \calS_{\zeta,j}^{\Lambda t^\rho}$  (resp. $\det \calS_{q,j}^{\Lambda t^\rho}$) counts the vanishing order of $\det \calS_{\zeta,j}^{\Lambda t^\rho}$ (resp. $\det \calS_{q,j}^{\Lambda t^\rho}$) at $t=1$. Note that 
		\[\val \left(\Lambda t^\rho([K_\alpha; (\rho, \alpha)-\frac{m}{2}(\alpha, \alpha)]_q)\right) = 1 \]
		if and only if  $\Lambda(K_{2\alpha})q^{2(\rho,\alpha)-m(\alpha,\alpha)} = 1$. 
		Now suppose that there exists some $\alpha\in \Phi^+$, such that for $1\leqslant m \leqslant j$,  $\Lambda(K_{2\alpha})q^{2(\rho,\alpha)-m(\alpha,\alpha)}\neq 1$. Then under specialization $q\mapsto \zeta$, there are  two possible cases as follows: 
		
		$(\textbf{a}).\;$ For all sufficiently large admissible $\ell$, the value $\Lambda_\zeta(K_{2\alpha})$ is not a root of unity. Then 
		\[ \Lambda_\zeta(K_{2\alpha}) \zeta^{2(\rho,\alpha)-m(\alpha,\alpha)}\neq 1 \]
		for $1\leqslant m \leqslant j < \ell$.
		
		$ (\textbf{b}).\;$ There are infinitely many admissible $\ell$ such that $\Lambda_\zeta(K_{2\alpha})$ is a root of unity. Then by Lemma \ref{rootunity}, for $\ell \gg 0$, 
		\[\Lambda_\zeta(K_{2\alpha}) = \zeta_{0,\alpha}\zeta^{m_\alpha} \]
		for some root of unity $\zeta_{0,\alpha}$ and $m_\alpha \in \bbZ$. Write $\zeta_{0,\alpha} = e^{\frac{2\pi\imath a_\alpha}{b_\alpha}}$ with $a_\alpha, b_\alpha\in \bbZ_+$, $\gcd(a_\alpha, b_\alpha)=1$, and $a_\alpha \leqslant b_\alpha$. Then $\Lambda_\zeta(K_{2\alpha}) \zeta^{2(\rho,\alpha)-m(\alpha,\alpha)}=e^{2\pi\imath\kappa_\alpha}$ with 
		\begin{equation}\label{power}
			\kappa_\alpha = \frac{a_\alpha}{b_\alpha} + \frac{m_\alpha+2(\rho,\alpha)-m(\alpha,\alpha)}{\ell}.
		\end{equation}
		Let $\Gamma\subset \Phi^+$ be the set of all roots in $\Phi^+$ satisfying the condition $(\textbf{b})$. For large enough  admissible $\ell$, we can find a positive integer $m_\Lambda$, e.g., $m_\Lambda = 2\max\{d_\alpha\mid\alpha\in \Gamma \}$, independent of $\ell$, such that for $\alpha\in \Gamma$, we have $0<\kappa_\alpha<1$ for all $1\leqslant m \leqslant j < \frac{\ell}{m_\Lambda}$, i.e., $\Lambda_\zeta(K_{2\alpha})\zeta^{2(\rho,\alpha)-m(\alpha,\alpha)}\neq 1$. This completes the proof of the claim. 
		
		Recall the Jantzen filtration introduced in the previous section. For each $j$-th graded component and each $i\in \bbN$, define 
		\[F^iM_\zeta(\Lambda)_j := F^iM_\zeta(\Lambda)\cap M_\zeta(\Lambda)_j \quad\text{and}\quad  F^iM_q(\Lambda)_j := F^iM_q(\Lambda)\cap M_q(\Lambda)_j. \]
		The valuation identity ~\eqref{valuation} implies that for any $j<\frac{\ell }{m_\Lambda}$, 
		\[\sum_{i= 1}^\infty\dim_\bbC F^iM_\zeta(\Lambda)_j = \sum_{i= 1}^\infty\dim_{\bfk} F^iM_q(\Lambda)_j. \]
		Since for each admissible $\ell$ we always have $\dim_{\bfk} F^iM_q(\Lambda)_j \geqslant \dim_\bbC F^iM_\zeta(\Lambda)_j $, it follows that the two sides must be equal term-by-term: 
		\[\dim_{\bfk} F^iM_q(\Lambda)_j = \dim_\bbC F^iM_\zeta(\Lambda)_j  \]
		for all $i$ and all $j< \frac{\ell }{m_\Lambda}$. In particular, this gives $\dim_{\bfk} F^1M_q(\Lambda)_j = \dim_\bbC F^1M_\zeta(\Lambda)_j$, which shows that the radicals of $\calS_{q,j}^\Lambda$ and of  $\calS_{\zeta, j}^\Lambda$ have the same dimension for all $j< \frac{\ell }{m_\Lambda}$. 
	\end{proof}
	\begin{remark}
		As shown in the proof (cf. Case $(\textbf{b})$), it is necessary for each $\ell$ to fix a specific primitive $\ell$-th root of unity $\zeta$ rather than choosing one arbitrarily. Otherwise, as the expression for $\kappa_\alpha$ in Equation ~\eqref{power} suggests, the choice of $m_\Lambda$ may depend on $\ell$.
	\end{remark}
	
	Now we proceed with the proof of Theorem \ref{mainthm}. 
	
	\begin{proof}[Proof of Theorem \ref{mainthm}]
		Consider the negative part $u_\zeta^-$ of the small quantum group $u_\zeta$, which is defined as the image of $U_\zeta^-$ in $u_\zeta$. Recall that $u_\zeta^-$ is of dimension $\ell^N$, with a basis given by 
		\[\left\{F^{\underline{m}}\mid \underline{m}\in \{0,1,\ldots, \ell-1 \}^N\right\}. \]
		The diagonal Verma module $\overline{M}_\zeta(\Lambda)$ is isomorphic to $u_\zeta^-$ as a $u_\zeta^-$-module. Since the simple $U_\zeta$-module $L_\zeta(\Lambda)$ is a quotient of $\overline{M}_\zeta(\Lambda)$, it follows that $\dim L_\zeta(\Lambda) \leqslant \dim \overline{M}_\zeta(\Lambda) = \dim u_\zeta^-$. By Lemma \ref{dimensionformula}, we can fix a positive integer $m_\Lambda$, independent of $\ell$, such that $\dim L_\zeta(\Lambda)_j = \dim_{\bbC(q)}L_q(\Lambda)_j$ for all $j < \frac{\ell}{m_\Lambda}$. This yields the inequality:
		\begin{equation}\label{lowerbound}
			\sum_{j<\ell/m_\Lambda}\dim_{\bbC(q)}L_q(\Lambda)_j = \sum_{j<\ell/m_\Lambda}\dim L_\zeta(\Lambda)_j \leqslant \dim L_\zeta(\Lambda).
		\end{equation}
		Let $\nu = \sum_{i=1}^nn_i\alpha_i\in Q^+$, and  define its height by $\height(\nu) = \sum_{i=1}^n n_i$. Then we compute 
		\[(2\ell-2)(\rho,\rho) = (\ell -1)\sum_{\alpha\in \Phi^+}(\rho,\alpha)\geqslant (\ell-1)\sum_{\alpha\in \Phi^+}\height(\alpha).\] 
		It follows that  $\overline{M}_\zeta(\Lambda)_j = 0$ for $j> (2\ell-2)(\rho,\rho)$, and hence $L_\zeta(\Lambda)_j = 0$ for such $j$.  For  $\underline{m}\in \{0,\ldots,\ell-1 \}^N$, define 
		$|\underline{m}| = m_1+\ldots +m_N$.
		Let $\alpha_0$ be the highest positive root in $\Phi^+$, and set $b=\lfloor\frac{\ell}{m_\Lambda\height(\alpha_0)}\rfloor$. We now consider the following subspace of $u_\zeta^-$: 
		\[(u_\zeta^-)^{\leqslant b} =\Span_\bbC\{F^{\underline{m}}\mid \underline{m}\in\{0,1,\ldots, \ell-1 \}^N, \text{ and } |\underline{m}|\leqslant b \}.  \]
		Its dimension is  given by the binomial coefficient $\binom{b+N}{N},$
		which is greater than $\frac{\ell^N}{C}$, where we set  $C=m_\Lambda \height(\alpha_0)N!$, a constant independent of $\ell$. Then 
		$\dim u_\zeta^- \leqslant C\dim (u_\zeta^-)^{\leqslant b}$. 
		Since the simple module $L_\zeta(\Lambda)$ is generated by the one-dimensional subspace $L_\zeta(\Lambda)_0$ under the action of $u_\zeta^-$, we have  
		\[ L_\zeta(\Lambda)= u_\zeta^-L_\zeta(\Lambda)_0 \supset (u_\zeta^-)^{\leqslant b}L_\zeta(\Lambda)_0.  \]
		It then follows that
		\[\dim L_\zeta(\Lambda) =\dim u_\zeta^-L_\zeta(\Lambda)_0 
		\leqslant C\dim(u_\zeta^-)^{\leqslant b}L_\zeta(\Lambda)_0. \]
		Moreover, since $(u_\zeta^-)^{\leqslant b}L_\zeta(\Lambda)_0\subset \oplus_{ j <{\ell}/{m_\Lambda}} L_\zeta(\Lambda)_j$, we deduce the upper bound:
		\begin{equation}\label{upbound}
			\dim L_\zeta(\Lambda) 
			\leqslant C\sum_{ j <{\ell}/{m_\Lambda}}\dim L_\zeta(\Lambda)_j 
			=C\sum_{ j <{\ell}/{m_\Lambda}}\dim_{\bbC(q)} L_q(\Lambda)_j.
		\end{equation}
		Finally, combining the inequalities~\eqref{lowerbound} and ~\eqref{upbound}, we obtain the desired result. 
	\end{proof}

	\section{Equivalence of categories}\label{Sec:Equiv}
	Let $\bbF$ be a field such that $\bbQ(q)\subset \bbF \subset \bfk$, and let $\catO_{q,\bbF}$ be the full subcategory of all finitely generated weight $U_\bbF$-modules on which $U_\bbF^+$ acts locally finitely. This section establishes  Soergel's structural description of the blocks of  $\catO_{q,\bbF}$.

	\subsection{Integral Weyl groups and root subsystems}\label{IntegralWeyl&subsys}
	Since the action of $W$ on general weights in $\rmT_\bbF$ can lead to inconsistencies in the signs, we adjust this using the action of $\bbZ_2^n$. In particular, we introduce a useful notation as follows. 
	\begin{definition}\label{modified+reflection}
		Let $\alpha \in \Phi$ and $\Lambda \in \rmT_\bbF$ satisfy $\Lambda(K_{2\alpha}) \in q^{\bbZ(\alpha,\alpha)}$.
		Define a \emph{modified reflection} $s_\alpha^\Lambda\in \Wch$ relative to $\Lambda$ and $\alpha$  as 
		\[ s_\alpha^\Lambda = (\sigma_{\alpha}^\Lambda, s_\alpha)\in  \bbZ_2^n\rtimes W, \]
		where $\sigma_{\alpha}^\Lambda$ is defined by
		$ \sigma_{\alpha}^\Lambda(K_\mu) = (-1)^{\inprod{\mu}{\alpha}}K_\mu$ for $\mu\in Q$ if  $\Lambda(K_\alpha)\in -q^{\bbZ(\alpha,\alpha)/2}$; 
		otherwise, $\sigma_{\alpha}^\Lambda(K_\mu) = K_\mu$ for $\mu\in Q$.
		
	\end{definition}

	Recall the subset $\calT_\Lambda$ in Definition \ref{Subset}. We have the following results.
	
	\begin{lemma}\label{ref-weight}
		Let $\Lambda\in \rmT_\bbF$ satisfy $\Lambda(K_{2\alpha}) \in q^{\bbZ(\alpha,\alpha)}$ for some $\alpha\in \Phi$. Then 
		\begin{enumerate}[{\rm (1)}]
			\item If $\Lambda^\prime = \Lambda q^{\mu}$ for $\mu\in P$, then $s_\alpha^{\Lambda^\prime} = s_\alpha^\Lambda $.
			\item If $\Lambda^\prime = \sigma\Lambda $ for $\sigma \in \bbZ_2^n$, then $s_\alpha^{\Lambda^\prime} = \sigma^\prime s_\alpha^\Lambda$, where $\sigma^\prime\in \bbZ_2^n$ is defined via $K_\mu \mapsto \sigma(\alpha)^{(\mu, \alpha^\vee)}K_\mu$.
			\item If there exists an integer $m\in \bbN_+$ such that $(m, \alpha)\in \calT_\Lambda$, then  $s_\alpha^\Lambda\cdot \Lambda = \Lambda q^{-m\alpha} $. 
		\end{enumerate}
	\end{lemma}
	\begin{proof}
		For (1), since $2(\mu, \alpha) = \inprod{\mu}{\alpha}(\alpha,\alpha) $ and $\inprod{\mu}{\alpha}\in \bbZ$, it follows that $\Lambda^\prime(K_{2\alpha})\in q^{\bbZ(\alpha,\alpha)}$, so $s_\alpha^{\Lambda^\prime}$ is well-defined. Note that $\sigma_{\alpha}^\Lambda=\sigma_\alpha^{\Lambda^\prime}$. Therefore,  $s_\alpha^\Lambda = s_\alpha^{\Lambda^\prime}$. For (2), note that $(\sigma\Lambda)(K_\alpha)\in -q^{\bbZ(\alpha, \alpha)/2}$ if and only if $\Lambda(K_\alpha)\in - \sigma(\alpha)q^{\bbZ(\alpha, \alpha)/2}$. Then by definition $\sigma_\alpha^{\Lambda^\prime}(K_\mu)= \sigma(\alpha)^{(\mu, \alpha^\vee)}\sigma_\alpha^{\Lambda}(K_\mu)$, and thus, it confirms Assertion (2). For (3),  it can be easily checked that $(\rho-s_\alpha\rho, \mu) = \inprod{\mu}{\alpha}(\rho, \alpha)$ for $\mu\in Q$. Then 
		{ \begin{align*}
			s_\alpha^\Lambda\cdot \Lambda(K_\mu) &=\left( q^{-\rho}s_\alpha^\Lambda(\Lambda q^\rho)\right)(K_\mu) \\
			&= q^{(s_\alpha\rho - \rho, \mu)}(s_\alpha\Lambda)(K_\mu)\sigma_{\alpha}^\Lambda(\mu) \\
			&=  q^{(s_\alpha\rho - \rho, \mu)}\Lambda(K_\alpha)^{-\inprod{\mu}{\alpha}}\Lambda(K_\mu)\sigma_{\alpha}^\Lambda(\mu) \\
			&=q^{(s_\alpha\rho - \rho, \mu)+\inprod{\mu}{\alpha}(\rho, \alpha) - \frac{m(\alpha,\alpha)}{2}\inprod{\mu}{\alpha}}\Lambda(K_\mu) \\
			&= (q^{-m\alpha}\Lambda)(K_\mu)
		\end{align*} }%
		for any $\mu\in Q$. Then Assertion (3) is proved as desired. 
	\end{proof}
	
	Recall the notions of \emph{integral Weyl group and integral root subsystem} $(W_\lambda, \Phi_\lambda)$ associated with $\lambda \in \frakh^*$, cf. \cite{Hum08}:
	\begin{align*}
		W_\lambda &= \{w\in W \;|\; w\lambda - \lambda \in Q \}, \\
		\Phi_\lambda &= \{\alpha \in \Phi \;|\; \inprod{\lambda}{\alpha} \in \bbZ \}.
	\end{align*} 
	
	For a weight $\Lambda\in \rmT_\bbF$ define a subset $\Phi_\Lambda$ of $\Phi$ as 
	\[\Phi_\Lambda = \{\alpha \in \Phi \;|\; \Lambda(K_{2\alpha}) \in q^{\bbZ(\alpha, \alpha)}  \}. \]
	A subset $\Psi\subset \Phi$ is called a \emph{root subsystem} of $\Phi$ if $-\alpha \in \Psi$ and $s_\beta \alpha \in \Psi$ if $\alpha, \beta \in \Psi$. One can check that  $\Phi_\Lambda$ is a root subsystem of $\Phi$. 
	
	Let $\Wch_\Lambda$ be the \emph{finite reflection subgroup} of $\Wch$ generated by all modified reflections $s_\alpha^\Lambda$, $\alpha\in \Phi_\Lambda$. Define $W_\Lambda$ as the image of $\Wch_\Lambda$ under the canonical projection $\Wch\rightarrow W$. Note that $\Wch_\Lambda \cap \bbZ_2^n$ is trivial, so the projection map $\Wch_\Lambda\rightarrow W_\Lambda$ is in fact an isomorphism of groups. It follows from Lemma \ref{ref-weight} that 
	
	\begin{corollary}\label{int-Weyl}
		For any weights $\Lambda, \Lambda^\prime \in \rmT_\bbF$, we have 
		\begin{enumerate}[{\rm (1)}]
			\item If $\Lambda^\prime = \Lambda q^\mu$ for some $\mu \in P$, then $\Phi_{\Lambda^\prime} = \Phi_{\Lambda}$ and $\Wch_{\Lambda}=\Wch_{\Lambda^\prime}$.
			\item If $\Lambda^\prime = \sigma\Lambda $ for some $\sigma \in \bbZ_2^n$, then $\Phi_{\Lambda^\prime} = \Phi_{\Lambda}$ and $\Wch_{\Lambda}\cong \Wch_{\Lambda^\prime}$.
		\end{enumerate} 
	\end{corollary}
	\begin{proof}
		(1) is clear. For (2), by definition, it is straightforward to verify that $\Phi_{\Lambda^\prime} = \Phi_{\Lambda}$, and hence $W_{\Lambda^\prime} = W_{\Lambda}$. Therefore, the above isomorphism follows, which sends the modified reflection $s_\alpha^\Lambda$ to $s_{\alpha}^{\Lambda^\prime}$ for $\alpha\in \Phi_\Lambda$.
	\end{proof}
	
	More generally, we can define $\Phi_\Lambda$ and $\Wch_\Lambda$ for weights $\Lambda\in \rmT_\bfk$ in the same way. Then we have the following result.
	
	\begin{proposition}\label{int-root-sys}
		For a linear weight $\Lambda = q^{\lambda}, \lambda \in \frakh_{\bbQ}^*$, we have 
		\begin{enumerate}[{\rm 1)}]
			\item $\Phi_\Lambda = \Phi_\lambda$ and $\Wch_\Lambda = W_\Lambda = W_\lambda$;
			
			\item $\Wch_\Lambda = \{w \in \Wch \;|\; w\Lambda = \Lambda q^\beta, \text{ for some } \beta \in Q \}$. 
		\end{enumerate}
	\end{proposition}
	\begin{proof}
		It is clear that $q^{2(\lambda,\alpha)} \in q^{\bbZ(\alpha, \alpha)}$ if and only if $2(\lambda, \alpha)\in \bbZ(\alpha, \alpha)$, i.e., $\inprod{\lambda}{\alpha} \in \bbZ$. For any $\check{w} \in \Wch$, write $\check{w} = \sigma w$ with $\sigma\in \bbZ_2^n, w\in W$. Suppose that $\check{w}\Lambda = \Lambda q^{\beta}$ for some $\beta \in Q$. The definition of $\Wch$-action on $\rmT_q$ yields that 
		$\check{w}\Lambda = \sigma q^{w\lambda} = q^{\lambda + \beta}$ which implies that $\sigma$ is the identity in $\bbZ_2^n$ and $w\lambda = \lambda + \beta$. As a result, $\Phi_\Lambda = \Phi_\lambda$ and $\Wch_\Lambda = W_\Lambda$.  By \cite[Theorem 3.4]{Hum08}, we get that $W_\lambda$ is an integral Weyl subgroup of $W$ generated by $s_\alpha$ for $\alpha \in \Phi_\lambda$. So $W_\lambda = W_\Lambda$. The second claim can be checked directly. 
	\end{proof}
	
	Define $\Phi^{\pm}_\Lambda = \Phi_\Lambda\cap\Phi^\pm$, and let $\Pi_\Lambda\subset \Phi^+_\Lambda$ be the set of simple roots of $\Phi_\Lambda$. Denote by $S_\Lambda = \{s_\alpha\mid \alpha\in \Pi_\Lambda \}$ the set of the corresponding simple reflections. Then $(W_\Lambda, S_\Lambda)$ forms a \emph{Coxeter group} associated with  $\Lambda\in \rmT_\bbF$. Similarly, $(W_\lambda, S_\lambda)$ is the corresponding Coxeter group associated with $\lambda\in \frakh^*$ or the linear weight $q^\lambda$ for $\lambda\in \frakh_{\bbQ}^*$.

	Denote $\Phi^\vee = \{\alpha^\vee = \frac{2\alpha}{(\alpha, \alpha)} \;|\; \alpha \in \Phi \}$. This is the \emph{dual system} of $\Phi$. We say a root  subsystem $\Psi \subset \Phi$ is \emph{closed} if $\alpha +\beta \in \Psi$ whenever $\alpha, \beta\in \Psi$ and $\alpha + \beta \in \Phi$. One can easily check the following result. 
	
	\begin{lemma}\label{dualsys}
		Let $\lambda \in \frakh^*$. Then $\Phi_\lambda^\vee$ is a closed root subsystem of $\Phi^\vee$. \qed 
	\end{lemma}
	\begin{remark}\label{Rmk-counterexample}
		Note that not every dual-closed root subsystem arises as $\Phi_\lambda$ for some weight $\lambda$. For instance, in the root system of type $B_3$, the subsystem consisting solely of the short roots provides a counterexample. 
	\end{remark}
	
	We end this subsection with an example where Lemma~\ref{dualsys} fails for the root subsystem $\Phi_\Lambda$ associated to a non-linear weight $\Lambda\in \rmT_\bbF$. 
	
	\begin{example}\label{Notclosed}
		Assume that $\Phi$ is the root system of type $B_2$ with simple roots $\alpha_1$ (long) and $\alpha_2$ (short). Set $\alpha = \alpha_1+\alpha_2$ and $\beta = \alpha_1+2\alpha_2$. Let $\Lambda\in \rmT_{\bbC(q)}$ satisfy that $\Lambda(K_{\alpha_1})=1$ and $ \Lambda(K_{\alpha_2})= \imath q\inv$.  
		Then a direct computation yields that 
		$\Phi_\Lambda = \{\pm \alpha_1, \pm \beta \}$, 
		which forms a root system of type $A_1\times A_1$. 
		Note that $\alpha= \alpha_1^\vee+ \beta^\vee \in \Phi^\vee$ while $\Lambda(K_{2\alpha}) = - q^{-2}\in -q^{\bbZ(\alpha, \alpha)}$, so the dual $\Phi_\Lambda^\vee$ is not closed. Hence $\Phi_\Lambda$ can not be of form $\Phi_\lambda$ for any weight $\lambda\in \frakh^*$ by Lemma \ref{dualsys}. 
		\qed
	\end{example}

	\subsection{Deformed category $\catO_R$}
	From now on, suppose that $R$ is a commutative Noetherian $U_\bbF^0$-algebra, equipped with a structure map $\tau: U_\bbF^0\rightarrow R$. We refer to such an algebra $R$ as a \emph{deformation algebra}. Recall the $R$-subalgebras $U_R^\pm$ and $U_R^0$ introduced in Subsection \ref{Specializations}. For  $\Lambda\in \rmT_\bbF$, we define the corresponding weight $\tau\Lambda\in \rmT_R$ as follows. Given an element $K = \sum_\mu K_\mu\otimes r_\mu$ in $U_R^0=U_\bbF^0\otimes_\bbF R$, where $r_\mu\in R$, we set 
	\[(\tau\Lambda)(K) = \sum_\mu\Lambda(K_\mu)\tau(K_\mu)r_\mu. \]
	This leads to the following notion of the deformed category $\catO_R$. 
	\begin{definition}
		Let $\catO_R$ denote the full subcategory of all $U_R$-modules $M$ satisfying the following conditions:
		\begin{enumerate}[{\rm (i)}]
			\item $M$ is finitely generated as a $U_R$-module.
			\item $M$ is a weight $U_R$-module with weight space decomposition $M = \oplus_{\Lambda\in \rmT_\bbF}M_{\tau\Lambda}$.
			\item For every $v\in M$, the $U_R^+$-module $U_R^+.v$ is finitely generated as an $R$-module. 
		\end{enumerate}
	\end{definition}
	
	\begin{remark}
		The category $\catO_R$ is an abelian, $R$-linear category. In the special case $R=\bbF$ with any $U_\bbF^0$-structure $\tau:U_{\bbF}^0\rightarrow \bbF$, i.e., $\tau\in \rmT_\bbF$, one has $\catO_R = \catO_{q,\bbF}$. Moreover, for each $M$ in $\catO_{R}$, its set of  weights satisfies 
		$\Omega(M) \subset \tau \rmT_\bbF$. 
	\end{remark}
	
	For any $\Lambda\in \rmT_\bbF$, the {deformed Verma module} $M_R(\tau\Lambda) $ belongs to $ \catO_R$. Note that $M_R(\tau\Lambda)$ is free as an $R$-module, and hence every module in $\catO_R$ with a \emph{Verma flag} (a finite filtration whose subquotients are isomorphic to Verma modules) is free over $R$. 
	For any morphism of deformation algebras $R \rightarrow R^\prime$, there is a \emph{base change functor }
	\[ \xymatrix{
		(\text{-}) \otimes_R R^\prime: \catO_R \ar[r] &  \catO_{R^\prime}.} \]
	In particular, one has a $U_{R^\prime}$-module isomorphism  $M_R(\tau\Lambda)\otimes_R R^\prime \cong M_{R^\prime}(\tau^\prime\Lambda)$, where $\tau^\prime: U_q^0 \rightarrow R^\prime$ is the structure map. 
	\vspace{.3cm}
	
	\emph{For convenience, we will write $\tau\Lambda$ simply as $\Lambda$ when referring to weights in $\rmT_R$, whenever no confusion arises.} With this convention, the notations $M_R(\Lambda)$ and $M_{R^\prime}(\Lambda)$ are well-defined for each $\Lambda\in \rmT_\bbF$.
	
	\vspace{.3cm}
	
	Let $R$ be a local deformation algebra with  maximal ideal $\frakm$, and let $\bbK= R/\frakm$ be its residue field. A standard deformation-theoretic argument shows that the base change functor $(\text{-}) \otimes_R\bbK$ induces a bijection between the following two isomorphism classes
	\[ 	\xymatrix{
		\big\{\text{simple objects in }  \catO_R\big\}  \ar@{<->}[r]^{1-1} &  \big\{\text{simple objects in } \catO_\bbK\big\}.}\]
	The simple modules of $\catO_\bbK$ (and hence of $\catO_R$) are parametrized by their highest weights, i.e., by elements of $\rmT_\bbF$. For $\Lambda\in \rmT_\bbF$, let $L_\bbK(\Lambda)$ be a simple module in $\catO_\bbK$ with highest weight $\Lambda$. It is a quotient of $M_\bbK(\Lambda)$.  
	
	\vspace{.3cm}
	
	The deformed category $\catO_R$ has enough projective objects for any deformation algebra $R$. We now gather several standard results from deformation theory for later use; see  \cite[Sections 3 and 4]{AJS94} for details. Let $P$ be a projective module in $\catO_R$. Then for any morphism $R\rightarrow R^\prime$ of deformation algebras, the module  $P\otimes_RR^\prime$ is projective in $\catO_{R^\prime}$, and the natural map 
	\begin{equation}\label{Proj-basechange}
		\xymatrix{
			\Hom_{\catO_R}(P, \text{-})\otimes_RR^\prime \ar[r] &  \Hom_{\catO_{R^\prime}}(P\otimes_R R^\prime, \text{-}\otimes_R R^\prime) }
	\end{equation}
	is an isomorphism of functors from $\catO_R$ to the category of $R^\prime$-modules. Suppose further that $R$ is a local deformation algebra with maximal ideal $\frakm$, and let $\bbK= R/\frakm$ denote its residue field. Then the functor $(\cdot) \otimes_R\bbK$ gives a bijection between the isomorphism classes of projective objects in  $\catO_R$ and those in $\catO_\bbK$. For each $\Lambda\in \rmT_\bbF$, let $P_R(\Lambda)$ denote the (indecomposable) \emph{projective cover} of $M_R(\Lambda)$ in $\catO_R$. Every projective object in $\catO_R$ is isomorphic to a direct sum of projective covers. 
	
	\subsection{Linkage principle and block decomposition}
	Let $R$ be a local deformation algebra with residue field $\bbK$. A \emph{partial order} $\uparrow_R$  on $\rmT_\bbF$  is generated by $\Lambda \uparrow_R \Lambda^\prime$ if $\Lambda = \Lambda^\prime$ or there exist an integer $m\in \bbN_+$ and a root $\alpha\in \Phi$ such that $(\tau\Lambda^\prime)(K_{2\alpha}) = q^{(-2\rho+m\alpha,\alpha)}$ in $\bbK$ and $\Lambda^\prime= \Lambda q^{m\alpha}$. Then we have the following argument.  
	
	\begin{lemma}\label{BGG-reciprocity}
		Let $R$ be a local deformation algebra with residue field $\bbK$. For any $\Lambda\in \rmT_\bbF$, the projective cover  $P_R(\Lambda)$ admits a Verma flag whose factors are of form $M_R(\Lambda^\prime)$ where $\Lambda^\prime\in \rmT_\bbF$ with $\Lambda^\prime  \Lambda\inv \in q^{Q_+}$. Any such $M_R(\Lambda^\prime)$ occurs exactly $[M_\bbK(\Lambda^\prime): L_\bbK(\Lambda)]$ times. Moreover,  $[M_\bbK(\Lambda^\prime): L_\bbK(\Lambda)]$ is non-zero if and only if $\Lambda \uparrow_R \Lambda^\prime$, or equivalently, $\Lambda \uparrow_\bbK \Lambda^\prime$. 
	\end{lemma}
	\begin{proof}
		The proof proceeds in the same way as in the root of unity case (see \cite[Lemma 2.16, Proposition 4.15]{AJS94}). By Appendix \ref{App-Tranfun} the tensor product $E\otimes M_R(\Lambda^\prime)$ for any finite dimensional module $E\in \catO_{q,\bbF}$ and any weight $\Lambda^\prime\in \rmT_\bbF$ belongs to $\catO_{R}$, which has a Verma flag. The (indecomposable) projective cover $P_R(\Lambda)$ occurs as a direct summand of such a tensor product for proper $E$ and $\Lambda^\prime$. Hence $P_R(\Lambda)$ also admits a Verma flag. The non-vanishing of multiplicities can be deduced inductively from the Jantzen sum formula~\eqref{Jantzenformula}.
	\end{proof}
	We say that two weights $\Lambda$ and $\Lambda^\prime$ are \emph{linked} over $R$ if $\Lambda \uparrow_R \Lambda^\prime$ or $\Lambda^\prime \uparrow_R \Lambda$, and write  $\Lambda\sim_R\Lambda^\prime$ in this case. Let $X\subset \rmT_\bbF/\sim_R$ be any union of \emph{linkage classes}. Define $\catO_{R, X}$ to be the full subcategory of $\catO_R$ consisting of all modules $M$ whose simple subquotients have highest weights lying in linkage classes from $X$. In particular, if $X$ consists of a single linkage classes, then $\catO_{R,X}$ is called a \emph{block} of $\catO_{R}$. We have the block decomposition of $\catO_R$:
	\[\catO_R = \bigoplus_{X\in \rmT_\bbF/\sim_R}\catO_{R, X}. \]
	For any morphism of local deformation algebras $R\rightarrow R^\prime$, the base change $(\cdot)\otimes_R R^\prime$ maps $\catO_{R,X}$ into $\catO_{R^\prime, X}$. Note that the linkage classes over $R^\prime$ are finer than those over $R$.  
	
	\begin{definition}\label{Deformedrootsys}
		Let $R$ be a local deformation algebra with  residue field $\bbK$, and let $\tau:U_\bbF^0\rightarrow R$ be the structure map. Define
		\[\Phi_{R,\Lambda}=\{\alpha\in \Phi\mid (\tau\Lambda)(K_{2\alpha})\in q^{\bbZ(\alpha,\alpha)} \text{ in\;} \bbK \}.\]
	\end{definition}
	
	In analogy with $\Phi_\Lambda$, the set $\Phi_{R, \Lambda}$ forms a root subsystem of $\Phi$. For each $\alpha\in \Phi_{R, \Lambda}$, let $s_{R,\alpha}^\Lambda$ denote the modified reflection $s_\alpha^{\tau\Lambda}$ from Definition \ref{modified+reflection} with $\Lambda$ replaced by $\tau\Lambda$. We then define  $\Wch_{R,\Lambda}\subset \Wch$ to be the subgroup generated by $s_{R,\alpha}^\Lambda$, $\alpha\in \Phi_{R,\Lambda}$, and let $W_{R,\Lambda}\subset W$ be its image under the projection $\Wch \rightarrow W$. 
	If $X\in \rmT_\bbF/\sim_R$ and $\Lambda\in X$, Lemma~\ref{ref-weight}(3) shows that $X=\Wch_{R,\Lambda}\cdot\Lambda$. We shall also write $\Phi_{R,X}$, $\Wch_{R,X}$ and $W_{R,X}$ for $\Phi_{R,\Lambda}$, $\Wch_{R,\Lambda}$ and $W_{R,\Lambda}$, respectively.
	
	Since $\Wch_{R,X}$ is finite, the set $X$ admits a unique highest (resp. lowest) weight with respect to the partial order defined by 
	$	\Lambda' \leq \Lambda''$  if $ \Lambda''(\Lambda')^{-1} \in q^{Q_+}$. 
	\begin{definition}
		Let $X\in \rmT_\bbF/\sim_R$. A weight $\Lambda\in X$ is called  \emph{dominant} (resp. \emph{antidominant}) if 
		$	\Lambda^\prime \leq \Lambda$ (resp. $\Lambda \leq \Lambda^\prime$ ) for all $\Lambda^\prime \in X$.
	\end{definition}
	
	\begin{remark}
		Suppose $\Lambda = q^\lambda$ for some $\lambda\in \frakh_{\bbQ}^*$. When $R=\bbF$ with the structure map $\tau: K_{\alpha_i}^{\pm1} \mapsto 1$, the condition that $\Lambda^\prime \leq \Lambda$ (resp. $\Lambda \leq \Lambda^\prime$) for all $\Lambda^\prime \in X$ is equivalent to $\inprod{\lambda+\rho}{\alpha} \notin -\bbN_+$ (resp. $\inprod{\lambda+\rho}{\alpha} \notin \bbN_+$) for all $\alpha\in \Phi^+$; we say $\lambda\in \frakh^*$ is dominant (resp. antidominant) if $\lambda$ satisfies the latter condition. Hence, the above  terminologies agree with those used in the classical case. 
	\end{remark}
	
	Let $\Lambda, \Lambda^\prime \in \rmT_{\bbF}$ be such that 
	$\Lambda^\prime\Lambda^{-1}$ is of the form $q^\nu$ for some $\nu\in P$; 
	we call such weights \emph{compatible}. Let $R$ be a local deformation algebra. 
	By Corollary~\ref{int-Weyl}, one then has 
	\[
	\Phi_{R,\Lambda} = \Phi_{R,\Lambda^\prime} 
	\quad \text{and} \quad 
	\Wch_{R,\Lambda} = \Wch_{R,\Lambda^\prime}.
	\]
	Choose $w\in W$ such that the $U_\bbF$-module $L_q(w\Gamma)$ is finite-dimensional, where $\Gamma :=\Lambda^\prime\Lambda^{-1}$. 
	Let 
	$X, X^\prime \in \rmT_\bbF/\sim_R$ be the linkage classes containing 
	$\Lambda$ and $\Lambda^\prime$, respectively. 
	For any module $M\in \catO_{R,X}$, we define the \emph{translation functor} $T_\Lambda^{\Lambda^\prime}$ from $\catO_{R,X}$ to $\catO_{R,X^\prime}$ as 
	\[
	T_\Lambda^{\Lambda^\prime}(M) 
	= \pr_{X^\prime}\!\left(L_q(w\Gamma)\otimes M\right),
	\]
	where $\pr_{X^\prime}$ denotes the projection functor 
	from $\catO_R$ onto the block $\catO_{R,X^\prime}$.
	Essentially, these functors behave as in the classical case. 
	Their exactness, biadjointness and further properties are summarized in Appendix~\ref{App-Tranfun}.
	\subsection{Deformed endomorphism algebras}\label{Deformed-End-Alg}
	Let $X\in \rmT_\bbF/\sim_R$. Suppose that $\Lambda\in X$ is antidominant, and set 
	\[\calZ_{R,X} := \End_{\catO_{R}}(P_R(\Lambda)), \]
	the endomorphism $R$-algebra of the antidominant projective cover $P_R(\Lambda)$. The following results describe the structure of this algebra in two special cases.

	\begin{lemma}[\!\!{\cite[Section 3.1]{Fie06}}]\label{Center-Localization}
		Let $R$ be a DVR with maximal ideal $\frakm$, and let $X\in \rmT_{\bbF}/\sim_R$. Then 
		\begin{enumerate}[{\rm (1)}]
			\item If $X$ consists of a single element, then the $R$-algebra $\calZ_{R,X}$ is isomorphic to $R$.
			\item If $X$ consists of two elements, then $\calZ_{R,X}$ is isomorphic to $R\times_\bbK R$, the fibre product (pullback) of two projection maps $R \rightarrow \bbK:=R/\frakm$.
			
		\end{enumerate}
	\end{lemma}
	\begin{proof}
		For (1), if $X$ consists of a single element $\Lambda$, which is both dominant and antidominant, then $P_R(\Lambda) \cong M_R(\Lambda)$, and hence $\calZ_{R,X}\cong R$. For (2), assume  $X=\{\Lambda,\Lambda^\prime \}$, where $\Lambda^\prime = s_\alpha^\Lambda\cdot \Lambda$ is dominant and $\Phi_{R,X}=\{\pm \alpha \}$. Recall that the bilinear form $\calS_R^{\Lambda^\prime}$ on $M_R(\Lambda^\prime)$ defines a decreasing sequence $\big\{F^iM_R(\Lambda^\prime) \big\}_{i\in \bbN}$ of $U_R$-submodules (see Subsection \ref{Janfiltration}). Applying $(\text{-})\otimes_R \bbK$, we obtain a corresponding filtration of $U_\bbK$-submodules. In this case, the right hand side of ~\eqref{Jantzenformula} is the character of $M_\bbK(\Lambda)$, hence 
		\[F^1M_R(\Lambda^\prime)\otimes_R \bbK \cong M_\bbK(\Lambda) \quad \text{and}\quad F^2M_R(\Lambda^\prime) \otimes_R\bbK = 0. \]
		Since $F^1M_R(\Lambda^\prime)\otimes_R \bbK$ is the unique maximal submodule of $M_\bbK(\Lambda^\prime)$, we have $[M_\bbK(\Lambda^\prime):L_\bbK(\Lambda)]=1$. Consequently, the multiplicity of $M_R(\Lambda^\prime)$ in $P_R(\Lambda)$ is one, and $P_R(\Lambda)$ admits a Verma flag with subquotients $M_R(\Lambda)$ and $M_R(\Lambda^\prime)$ (see Lemma \ref{BGG-reciprocity}). Since the dominant Verma module must appear as a submodule, we obtain the short exact sequence:
		\begin{equation*}
			\xymatrix@C=1.5em{
				0 \ar[r] &  M_R(\Lambda^\prime) \ar[r] & P_R(\Lambda) \ar[r]  & M_R(\Lambda)\ar[r] & 0.
			}
		\end{equation*}
		Let $\iota: M_R(\Lambda^\prime)\rightarrow P_R(\Lambda)$ denote the inclusion, and let $\iota^\prime: M_\bbK(\Lambda^\prime) \rightarrow P_\bbK(\Lambda)$ be its base change. Over the closed point, there is a canonical map $\pi^\prime: P_\bbK(\Lambda)\twoheadrightarrow M_\bbK(\Lambda) \hookrightarrow M_\bbK(\Lambda^\prime)$, which, by ~\eqref{Proj-basechange}, lifts to a morphism $\pi: P_R(\Lambda) \rightarrow M_R(\Lambda^\prime)$. Weight consideration gives $\pi^\prime\circ \iota^\prime = 0$, while $\pi\circ\iota $ lies in $ \End_{\catO_{R}}(M_R(\Lambda^\prime)) \cong R$ and  hence equals a multiplication by some  $x_\alpha\in \frakm$. If $\mathrm{val}_\frakm(x_\alpha) \geqslant 2$, then $\mathrm{Im }(\pi) \subset F^2M_R(\Lambda^\prime)$, so 
		\[\mathrm{Im }( \pi^\prime) = \mathrm{Im} (\pi)\otimes_R \bbK \subset F^2M_R(\Lambda^\prime)\otimes_R\bbK =0, \]
		contradicting $\pi^\prime \neq 0$. Also, $x_\alpha\neq 0$ since $F^1M_R(\Lambda^\prime)\neq 0$.  Thus $\mathrm{val}_\frakm(x_\alpha) = 1$; in particular, $\frakm$ is the principal ideal generated by $x_\alpha$. 
		Finally, for $f\in \calZ_{R,X}$, note that  $f(M_R(\Lambda^\prime)) \subset M_R(\Lambda^\prime)$ since $\Lambda^\prime > \Lambda$, hence $f$ induces scalars  $r_{\Lambda^\prime}\in \End_{\catO_{R}}(M_R(\Lambda^\prime))\cong R$ and $r_{\Lambda}\in \End_{\catO_{R}}(M_R(\Lambda))\cong R$. This yields 
		a natural map 
		\begin{equation*}
			\xymatrix@C=1.5pc{
				\calZ_{R,X} \ar[r] &  \End_{\catO_{R}}\left(M_R(\Lambda^\prime)\right)\oplus \End_{\catO_{R}}\left(M_R(\Lambda)\right) \ar[r]^-{\widesim{}} & R\oplus R.
			}
		\end{equation*}
		Since $x_\alpha\neq 0$, this map is injective, and its image is $\{(r_{\Lambda^\prime}, r_\Lambda)\mid r_{\Lambda}\equiv r_{\Lambda^\prime}\mod x_\alpha \}$; it sends $\iota\circ\pi \in \calZ_{R,X}$ to $(x_\alpha, 0)$ in $R\oplus R$. Thus $\calZ_{R,X}\cong R\times_\bbK R$. 
	\end{proof}
	
	\begin{remark}
		\begin{enumerate}[{\rm (i)}]
			\item In cases (1) and (2), $\catO_{R,X}$ is equivalent either to the category of finitely generated $R$-modules, or to the deformed principal block of $\mathfrak{sl}_2(\bbF)$ (or  $U_{q,\bbF}(\mathfrak{sl}_2)$). 
			\item The algebra $\calZ_{R,X}$ is the deformed version of Soergel’s coinvariant algebra for $W_{R,X}$ (see \cite[Endomorphismensatz]{Soe90}). As shown by Fiebig (loc. cit.), this algebra is essentially the categorical (or Bernstein) center of $\catO_{R,X}$.
		\end{enumerate}
		%
	\end{remark}
	
	Let $\mathrm{Frac}(U^0_\bbF)$ denote the fraction field of $U_\bbF^0$. Consider a localization $R$ of $U_\bbF^0$ at a prime ideal of $U_\bbF^0$. We regard the inclusion $U_\bbF^0 \stackrel{\tau}{\hookrightarrow} R \subset \mathrm{Frac}(U^0_\bbF)$ as the structure map for the local deformation algebra $R$. For a prime ideal $\frakp \subset R$ of height one, suppose that there exists $\alpha \in \Phi$ such that 
	\[\tau(K_{2\alpha})+a \in \frakp \] 
	for some $a\in \bbF$. Then $\frakp$ is a principal ideal generated by a prime factor of $\tau(K_{2\alpha})+a$ since $R$ is a Noetherian UFD (see \cite[Theorem 20.1]{Mat87}). Consequently, $R_\frakp$ is a DVR, and any linkage class in $\rmT_\bbF/\sim_{R_\frakp}$ consists of either one or two elements. As $R$ is an integrally-closed domain,  \cite[Theorem 11.5]{Mat87} implies that $R$ is the intersection of all its localizations at prime ideals of height one. Therefore, for $X\in \rmT_\bbF/\sim_{R}$ with $\Lambda \in X$ antidominant, we have 
	{\small \begin{align*}
			\calZ_{R,X} &\cong \End_{\catO_{R}}(P_R(\Lambda)) \otimes_R \bigcap_{\mathrm{ht}\frakp =1}R_\frakp \\
			&\cong \bigcap_{\mathrm{ht}\frakp =1}\End_{\catO_{R}}(P_R(\Lambda)) \otimes_RR_\frakp\\
			&\cong \bigcap_{\mathrm{ht}\frakp =1}\End_{\catO_{R_\frakp}}(P_{R_\frakp}(\Lambda)).
	\end{align*}}%
	Note that $\Lambda $ remains antidominant in its linkage class of  $\rmT_\bbF/\sim_{R_\frakp}$ for height-one prime $\frakp\subset R$. Hence, $\calZ_{R,X}$ (and thus $\calZ_{R,X}\otimes_R\bbK$) admits a combinatorial description by Lemma \ref{Center-Localization}. More precisely, we can get, under this description, that $\calZ_{R,X}$ consists of tuples $z=(z_{w})$ with entries in the direct sum $ \bigoplus_{w} R$, indexed by $w\in \Wch_{R,\Lambda}/\Stab_R(\Lambda)$ (equivalently, by weights in $X$), satisfying
	$z_{w} \equiv z_{s_{R,\alpha}^{\Lambda}w} \pmod{x_\alpha}$ for all $\alpha\in \Phi_{R,X}$. In particular, this implies that  $\calZ_{R,X}$ is a commutative $R$-algebra. Here we denote 
	\[ \Stab_R(\Lambda) = \Stab_{\Wch_{R,\Lambda}}(\Lambda):=\{w\in \Wch_{R,\Lambda}\mid w\cdot\Lambda = \Lambda\}\subset  \Wch_{R,\Lambda}. \] 
	
	Let $\Lambda, \Lambda^\prime\in \rmT_\bbF$ be two compatible weights. We say that $\Lambda, \Lambda^\prime$ lie in \emph{the closure of the same ``facet''} for $\Wch_{R,\Lambda}$ if they satisfy 
	\begin{equation}\label{WeylCha}
		q^{2(\rho, \alpha)}(\tau\Lambda)(K_{2\alpha})\in q^{\bbN(\alpha,\alpha)} \quad \text{iff}\quad q^{2(\rho, \alpha)}(\tau\Lambda^\prime)(K_{2\alpha})\in q^{\bbN(\alpha,\alpha)}\quad \text{for all } \alpha\in \Phi_{R,\Lambda}.
	\end{equation}
	Then we have the following result.
	
	\begin{proposition}\label{EndAlg-Antidom-Projs}
		Let $R$ be as above and let $X, X^\prime \in \rmT_\bbF/\sim_{R}$ with antidominant weights $\Lambda\in X$, $\Lambda^\prime \in X^\prime$. If $\Lambda, \Lambda^\prime$ are compatible and satisfy the condition \eqref{WeylCha} and  $\Stab_R(\Lambda) \subset \Stab_R(\Lambda^\prime)$, then 
		\[ 	\xymatrix{
			\calZ_{R,X} \ar[r]^-{\widesim{}} &  \calZ_{R,X^\prime}^{\bigoplus |\Stab_R(\Lambda^\prime)/\Stab_R(\Lambda)|}	}
		\]
		as $\calZ_{R,X^\prime}$-modules. Moreover, there exists a natural embedding of $R$-algebras $\calZ_{R,X^\prime} \rightarrow \calZ_{R,X}$ sending each $z\in \calZ_{R,X^\prime}$ to the  $|\Stab_R(\Lambda^\prime)/\Stab_R(\Lambda)|$-copies of $z$ in $\calZ_{R,X}$, in the sense of the above isomorphism.
	\end{proposition}
	\begin{proof}
		We first note the natural isomorphism from 
		$\Hom_{\catO_\bbK}\big(T_{\Lambda^\prime}^\Lambda \big(P_\bbK(\Lambda^\prime)\big), L_\bbK(w\cdot\Lambda)\big)$  to $ \Hom_{\catO_\bbK}\big(P_\bbK(\Lambda^\prime), T^{\Lambda^\prime}_\Lambda \big(L_\bbK(w\cdot\Lambda)\big)\big)$. 
		The latter space is one-dimensional over $\bbK$ when $w$ lies in $\Stab_R(\Lambda)$, and vanishes otherwise, because $T^{\Lambda^\prime}_\Lambda \big(L_\bbK(\Lambda)\big) \cong L_\bbK(\Lambda^\prime)$, a fact that follows from the integral part $\hat{\lambda}^\prime$ of $\Lambda^\prime$ lying in the \emph{upper closure} of the facet containing the integral part $\hat{\lambda}$ of $\Lambda$ (see Lemma \ref{Tran-Verma} and \cite[Theorem 7.11]{Hum08}). Since $T_{\Lambda^\prime}^{\Lambda} \big(P_R(\Lambda^\prime)\big)$ is projective, the above computation shows that it is exactly $P_R(\Lambda)$.
		Then $T_{\Lambda^\prime}^\Lambda$ induces an $R$-algebra homomorphism
		$\calZ_{R,X^\prime} \rightarrow \calZ_{R,X}$.
		Moreover, by Corollary \ref{App-tran-proj}(1), we have 
		{\small \[
		\begin{aligned}
			\calZ_{R,X}
			&\cong \Hom_{\catO_R}\big( T_{\Lambda^{\prime}}^\Lambda( P_R(\Lambda^\prime) ), P_R(\Lambda) \big)\\
			&\cong \Hom_{\catO_R}\big( P_R(\Lambda^\prime), T_{\Lambda}^{\Lambda^\prime}( P_R(\Lambda) ) \big)\\
			&\cong \Hom_{\catO_R}\big( P_R(\Lambda^\prime), P_R(\Lambda^\prime)^{\bigoplus r} \big)\\
			&\cong \calZ_{R,X^\prime}^{\bigoplus r}
		\end{aligned}
		\]}%
		where $ r := |\Stab_R(\Lambda^\prime)/\Stab_R(\Lambda)|$.
		Under the above morphism, each $z \in \calZ_{R,X^\prime}$ maps to the $r$-copies of $z$ in $\calZ_{R,X}$. Hence the morphism is injective. 
	\end{proof}
	
	\begin{remark}
		The description of $\calZ_{R,X}$ as tuples 
		$z = (z_w) \in \bigoplus_w R$
		naturally yields an action of the quotient group 
		$\Wch_{R,\Lambda}/\Stab_R(\Lambda)$
		defined by
		\[
		x.z = (z'_w), \qquad 
		z'_w = z_{w x}, \quad 
		\forall x \in \Wch_{R,\Lambda}/\Stab_R(\Lambda).
		\]
		Under this action, the subalgebra $\calZ_{R,X'}$  in the above proposition 
		can be identified with the 
		$\Stab_R(\Lambda^\prime)/\Stab_R(\Lambda)$\nobreakdash-invariant 
		subalgebra of $\calZ_{R,X}$.
	\end{remark}

	
	
	\subsection{The combinatorics of category $\catO_{q,\bbF}$}\label{Comb-catO}
	In order to determine the categorical structure of each block $\catO_{R, X}$, we study the endomorphism algebra for a projective generator in $\catO_{R,X}$. Specifically, consider the \emph{minimal projective generator} $P_{R,X} = \oplus_{\Lambda\in X}P_R(\Lambda)$  in $\catO_{R, X}$, and set $E_{R,X}=\End_{\catO_{R}}(P_{R,X})$. Then the functor 
	\begin{equation}\label{Hom-functor}
		\xymatrix{
			\Hom_{\catO_R}(P_{R,X},\text{-}) : \catO_{R, X} \ar[r] &  \mathrm{mod}\text{-}E_{R,X} }
	\end{equation}
	is an equivalence of $R$-linear categories, where $\mathrm{mod}\text{-}E_{R,X}$ denotes the category of finitely generated right $E_{R,X}$-modules. This is a standard argument; see, for instance, \cite[Theorem II.1.3]{Bas68} for more general setting. 
	
	\vspace{.3cm}
	Fix $\calR$ to be the localization of $U_{\bbF}^0$ at the maximal ideal generated by $K_\alpha-1$ for all $\alpha\in \Phi$. We refer to such an $\calR$ as being of \emph{type $\mathbf{1}$}. The canonical embedding $\tau: U_\bbF^0 \hookrightarrow \calR$ serves as the structure map, and the composition  $U_\bbF^0\hookrightarrow \calR \rightarrow \bbF$ is given by the evaluation $K_\alpha \mapsto 1$ for all $\alpha\in \Phi$. Throughout this subsection, we regard $\bbF$ as the residue field of $\calR$. Then $\catO_{\bbF} $ coincides with $ \catO_{q,\bbF}$, and for any $\Lambda\in \rmT_\bbF$, one has
	\[ 	\Phi_{\calR,\Lambda}=\Phi_\Lambda, \quad    \Wch_{\calR,\Lambda}=\Wch_\Lambda.  \]
	
	For any prime ideal $\frakp\subset \calR$, let $\calR_\frakp$ denote the localization of $\calR$ at $\frakp$, and let $\bbK_\frakp = \calR_\frakp/\frakp \calR_\frakp$ be its residue field.  For each $\alpha\in \Phi$, denote by $\calR_\alpha$ the localization of $\calR$ at the prime ideal generated by $\tau(K_\alpha)-1$ with residue field $\bbK_\alpha$. 
	\begin{lemma}\label{localization-split}
		Let $\Lambda\in \rmT_\bbF$, and let $\frakp\subset \calR$ be any prime ideal. Then a root $\alpha\in \Phi_{\calR_\frakp,\Lambda}$ if and only if $\alpha\in \Phi_\Lambda$ and  $\tau(K_\alpha) - 1\in \frakp$.  In particular,  $\Phi_{\calR_\alpha,\Lambda}$ is either empty or equal to $\{\pm \alpha \}$, and the corresponding linkage class is either $\{\Lambda \}$ or $\{\Lambda, s_\alpha^\Lambda\cdot \Lambda \}$. 
	\end{lemma}
	\begin{proof}
		By assumption, $\calR$ is a local algebra with maximal ideal generated by $\tau(K_\mu)-1$ for all $\mu\in P$.
		By Definition \ref{Deformedrootsys}, $\alpha\in \Phi_{R_\frakp, \Lambda}$ if and only if $(\tau\Lambda)(K_{2\alpha})\in q^{\bbZ(\alpha,\alpha)}$ in $\bbK_\frakp$.
		In this case, we have $\tau(K_{2\alpha})=1$ in $\bbK_\frakp$ and $\Lambda(K_{2\alpha})\in q^{\bbZ(\alpha,\alpha)}$,
		so $\tau(K_{2\alpha})-1\in \frakp$, $\alpha\in \Phi_\Lambda$.
	\end{proof}

	
	Recall the Coxeter system $(W_\Lambda, S_\Lambda)$ associated with $\Lambda\in \rmT_\bbF$ (see Subsection \ref{IntegralWeyl&subsys}), and set  $\overline{\Lambda}=\Wch_\Lambda\cdot\Lambda$. We shall show that the block $\calO_{\bbF,\overline{\Lambda}}$ depends only on the isomorphism class of this Coxeter system together with its action on $\Lambda$.
	
	\vspace{.3cm}
	For later use, let us first review some basic notions and results in $\catO_\bbF$.
	\begin{definition}
		A weight $\Lambda$ is said to be \emph{regular}  if  $\Stab(\Lambda):=\Stab_{\Wch_{\Lambda}}(\Lambda)$ is trivial.  
	\end{definition}

	The following fact is straightforward. 
	\begin{lemma}\label{uniq2reg}
		Assume that $w$ is the  unique element in $ \Wch_\Lambda$ such that $w\cdot \Lambda$ is dominant (or antidominant). Then $\Lambda$ is regular. \qed 
	\end{lemma}

	For any $\Lambda\in \rmT_\bbF$, we say that $\Lambda$ is an \emph{integral weight} if $\Phi_\Lambda = \Phi$; in this case,  $W_\Lambda=W$. It is clear that for any $\lambda\in P$, the linear weight $q^\lambda$ is integral and only these linear weights are integral. Recall the action of $\bbZ_2^n$ on $\rmT_\bbF$ (cf. \eqref{Z2action} and \eqref{Waction}). Then the semidirect product subgroup 
\[\calP :=  \bbZ_2^n \rtimes\{q^\lambda \;|\; \lambda\in P \} \]  
	of $\rmT_\bbF$ are exactly all integral weights in $\rmT_\bbF$. Moreover, the subset $ \bbZ_2^n \rtimes\{q^\lambda \;|\; \lambda\in P^+ \}$ of $\calP$, say $\calP^+$, parametrizes all isomorphism classes of finite dimensional simple $U_\bbF$-modules in terms of their highest weights \cite{Jan95}. For instance, 
	each weight $\sigma\in \bbZ_2^n \subset \calP^+$ corresponds to a one dimensional $U_\bbF$-module, on which  $K_{\alpha_i}^{\pm 1}= \sigma(\alpha_i)$, and $E_i=F_i=0$, $i\in I$. As a consequence, it induces an equivalence of categories
	\begin{align}\label{block2block}
		\xymatrix{
			\calF_\sigma: \catO_{\bbF, \overline{\sigma\Lambda}} \ar[r]^-{\widesim{}} &  \catO_{\bbF,\overline{\Lambda}}	}
	\end{align}
	via sending each module $M$ to the tensor product of $M$ and the one dimensional $U_\bbF$-module determined by $\sigma$ as above. 
	
	\vspace{.3cm}
	For simplicity, write $\sim$ for $\sim_{\calR} = \sim_\bbF$. Let $X\in \rmT_\bbF/\sim$, and choose $\Lambda\in X$ to be antidominant. We have already known that the algebra $\calZ_{\calR, X}$ defined in the previous subsection is a commutative $\calR$-algebra. Define the \emph{structure functor} 
	\[
	\xymatrix{
		\bbV_{\calR,X}: = \Hom_{\catO_R}(P_\calR(\Lambda), \text{-}): \catO_{\calR,X} \ar[r] &  \calZ_{\calR, X}\text{-}\mathrm{mod},}  \]
	which is the analogue of the classical functor considered in \cite{Soe90}.
	We omit a detailed discussion of its properties here, as they are parallel to those in the classical setting, and refer the reader to Appendix~\ref{App-Structure-functor} for further properties.
	
	\vspace{.3cm}
	Let $\calM_\calR$ be the full subcategory of $\catO_{\calR}$ consisting of modules which admits a Verma flag. In particular, all projectives in $\catO_{\calR}$ belong to $\calM_\calR$. Let $\calM_{\calR,X}$ be the corresponding block in $\calM_{\calR}$.	Let $X, X^\prime \in \rmT_\bbF/\sim$ with antidominant weights $\Lambda\in X$, $\Lambda^\prime \in X^\prime$. Assume that $\Lambda, \Lambda^\prime$ are compatible,  satisfy the condition \eqref{WeylCha} and that $\Stab(\Lambda) \subset \Stab(\Lambda^\prime)$. Then  Proposition \ref{EndAlg-Antidom-Projs} implies that $\calZ_{\calR,X^\prime}$ can be naturally embedded into $\calZ_{\calR,X}$. 
	Thus, we may define the induction functor $\Ind(\text{-}) = \Ind^{\calZ_{\calR,\Lambda}}_{\calZ_{\calR,\Lambda^\prime}}(\text{-}) =\calZ_{\calR,\Lambda}\otimes_{\calZ_{\calR,\Lambda^\prime}} (\text{-})$ and the restriction functor $\Res(\text{-})= \Res_{\calZ_{\calR,\Lambda^\prime}}^{\calZ_{\calR,\Lambda}}(\text{-})$, which is the right adjoint to $\Ind(\text{-}) $.
	
	\begin{lemma}\label{isoms-of-funtors}
		Let $X, X^\prime\in \rmT_\bbF/\sim$ with $\Lambda\in X$, $\Lambda^\prime\in X^\prime$ antidominant. Assume that $\Lambda, \Lambda^\prime$ are compatible,  satisfy the condition \eqref{WeylCha} and that $\Stab(\Lambda) \subset \Stab(\Lambda^\prime)$. Then we have the following isomorphisms of functors:
		\begin{enumerate}[{\rm (1)}]
			\item $\xymatrix{\bbV_{\calR,X^\prime}\circ T_\Lambda^{\Lambda^\prime} \cong \Res\circ\bbV_{\calR,X}: \calM_{\calR,X} \ar[r] &  \calZ_{\calR, X^\prime}\text{-}\mathrm{mod},}$
			\item $\xymatrix{\bbV_{\calR,X}\circ T_{\Lambda^\prime}^{\Lambda} \cong \Ind\circ\bbV_{\calR,X^\prime}: \calM_{\calR,X^\prime} \ar[r] &  \calZ_{\calR, X}\text{-}\mathrm{mod}.}$ 
		\end{enumerate}
	\end{lemma}
	\begin{proof}
		The key case in which $X$ consists of two elements and $X^\prime$ consists of one element is treated in Proposition \ref{Subgeneric-cases}(2). Thanks to Lemma \ref{localization-split}, every linkage class $X=\Wch_{R,\Lambda}\cdot\Lambda$ under $\sim$ splits into linkage classes under $\sim_{\mathcal{R}_{\alpha}}$, each containing one or two elements. Using the fact that objects in $\mathcal{M}_{\mathcal{R}}$ are  intersections of their localizations at height-one prime ideals (since they are free $\mathcal{R}$-modules; see the similar discussion in Subsection \ref{Deformed-End-Alg}), we prove the first assertion;  the second is similar. For $M\in \calM_{\calR,X}$, the $\calZ_{\calR, X^\prime}$-module $\bbV_{\calR,X^\prime}\big( T_\Lambda^{\Lambda^\prime}(M)\big)$ is also free as a $\calR$-module by Corollary \ref{prop-V}(1). For each prime $\frakp$, let $X_\frakp\subset X$ (resp. $X_\frakp^\prime\subset X^\prime$) be the linkage class with respect to $\sim_{\calR_\frakp}$ containing the weight $\Lambda$ (resp. $\Lambda^\prime$). Set $M_\frakp := M\otimes_\calR\calR_\frakp$. Then by Proposition \ref{prop-T}, Corollary \ref{prop-V}(2) and Proposition \ref{Subgeneric-cases}(2), we obtain  
		{\small	\[
			\begin{aligned}
				\bbV_{\calR,X^\prime}\big( T_\Lambda^{\Lambda^\prime}(M)\big) 
				&= \bigcap_{\mathrm{ht}\frakp =1}\bbV_{\calR,X^\prime}\big( T_\Lambda^{\Lambda^\prime}(M)\big)\otimes_{\calR}\calR_\frakp \cong \bigcap_{\mathrm{ht}\frakp =1}\bbV_{\calR_\frakp,X_\frakp^\prime}\big( T_\Lambda^{\Lambda^\prime}(M_\frakp)\big) \\
				&\cong \bigcap_{\mathrm{ht}\frakp =1} \Res\big(\bbV_{\calR_\frakp,X_\frakp}(M_\frakp)\big)\cong \Res\big(\bbV_{\calR,X}(M)\big).
			\end{aligned}
			\]}%
		This identification is functorial and thus yields the isomorphism in (1).
	\end{proof}
	Let $X^\prime \in \rmT_\bbF/\sim$. Take any  $\Lambda^\prime\in X^\prime$. From our previous discussion, we already know the algebraic structure of $\calZ_{\calR, X^\prime}$ (Subsection \ref{Deformed-End-Alg}). We now aim to describe the structure of the $\calZ_{\calR, X^\prime}$-module $\bbV P_\calR(\Lambda^\prime)$. 
	If $\Lambda^\prime$ is dominant, then we have an isomorphism $P_\calR(\Lambda^\prime)\cong M_\calR(\Lambda^\prime)$, and consequently $\bbV P_\calR(\Lambda^\prime)\cong \calZ_{\calR,X^\prime}/\calI_{\Lambda^\prime}$ by Lemma \ref{propertiesforfunctorV}(2). In general, for an arbitrary weight $\Lambda^\prime\in X^\prime$, the $\calZ_{\calR, X^\prime}$-module $\bbV P_\calR(\Lambda^\prime)$ can also be described inductively, based on the previous results and the discussion in Appendix \ref{Appendix}. This inductive description shows that the module structure depends only on the associated Coxeter group and its action on $\Lambda^\prime$.
	We outline the construction as follows. 
	
		(i) Let $\Lambda^\prime \in X^\prime$ be dominant. Choose a regular dominant weight $\Lambda\in \rmT_\bbF$ (with respect to $\sim$) that is compatible with $\Lambda^\prime$ and satisfies the condition \eqref{WeylCha}, and let $X \in \rmT_\bbF/\sim$ be the corresponding linkage class. Then  $T_\Lambda^{\Lambda^\prime}:\catO_{\calR,X}\rightarrow \catO_{\calR,X^\prime}$
		induces an embedding of $\calR$-algebras
		\[ \xymatrix{\calZ_{\calR,X^\prime}\ar[r] & \calZ_{\calR,X}} \]
		and $\calZ_{\calR,X^\prime}$ can be realized as the invariant subalgebra $\calZ_{\calR,X}^{\Stab(\Lambda^\prime)}$  (Proposition \ref{EndAlg-Antidom-Projs}). 
		
		(ii) Let $w\in \Wch_{\Lambda^\prime}$, and denote by $\bar{w}\in \Wch_{\Lambda^\prime}/\Stab(\Lambda^\prime)$ its image. Then the module $T_\Lambda^{\Lambda^\prime}P_\calR(w\cdot\Lambda)$ decomposes into $|\Stab(\Lambda^\prime)|$ copies of $P_\calR(\bar{w}\cdot\Lambda^\prime)$,  possibly together with additional summands $P_\calR(\bar{w}'\cdot\Lambda^\prime)$ with $\bar{w}'\cdot\Lambda^\prime>\bar{w}\cdot\Lambda^\prime$. 
		Applying the functor $\bbV$ yields that 
		\begin{equation*}
			\bbV_{\calR,X^\prime}\circ T_\Lambda^{\Lambda^\prime}\big(P_\calR(w\cdot\Lambda)\big) \cong \Res\circ\bbV_{\calR,X}\big(P_\calR(w\cdot\Lambda)\big). 
		\end{equation*}
		Accordingly, $	\bbV\big(T_\Lambda^{\Lambda^\prime}P_\calR(w\cdot\Lambda)\big) $ decomposes into $|\Stab(\Lambda^\prime)|$ copies of $\bbV P_\calR(\bar{w}\cdot\Lambda^\prime)$, and possibly additional terms $\bbV P_\calR(\bar{w}'\cdot\Lambda^\prime)$ with $\bar{w}'\cdot\Lambda^\prime>\bar{w}\cdot\Lambda^\prime$ (Lemma \ref{isoms-of-funtors} and Corollary \ref{App-tran-proj}).  Therefore, the above isomorphism provides an inductive description of $\bbV P_\calR(\bar{w}\cdot\Lambda^\prime)$ assuming that $\bbV P_\calR(w\cdot\Lambda)$ is already known  for all $w'\in \Wch_\Lambda(=\Wch_{\Lambda^\prime})$ with $w'\cdot\Lambda > w\cdot\Lambda$. 
		
		(iii) Using \eqref{block2block} we may assume without loss of generality that $\Wch_\Lambda = W_\Lambda$. Choose a reduced expression $w = s_{\gamma_1}\cdots s_{\gamma_t}$ with respect to the Coxeter group $(W_\Lambda, S_\Lambda)$, where $s_{\gamma_i}\in S_\Lambda$. We then  construct $\bbV P_\calR(w\cdot\Lambda)$ inductively on the length of $w$. Denote by $\Theta_i$ the translation through the ``$s_{\gamma_i}$-wall'' (see Appendix \ref{App-Tranfun}). Consider the $\calZ_{\calR,X}$-module   $\Theta_{t}\cdots\Theta_{1}P_\calR(\Lambda)$. The projective $P_\calR(w\cdot\Lambda)$ is uniquely determined as the indecomposable summand not isomorphic to any $P_\calR(w'\cdot\Lambda)$ with $l(w')<l(w)$. Similarly,  $\bbV P_\calR(w\cdot\Lambda)$ can be realized as the direct summand of 
		\[\calZ\otimes_{\calZ_t}\calZ\cdots\otimes_{\calZ_2}\calZ\otimes_{\calZ_1}\calZ/\calI_\Lambda, \]
		which is not isomorphic to $\bbV P_\calR(w'\cdot\Lambda)$ for any $w'\in W_\Lambda$ of smaller length than $w$ (cf. Lemma \ref{propertiesforfunctorV}). Here $\calZ := \calZ_{\calR,X}$ and $\calZ_i := \calZ_{\calR,X}^{\langle s_{\gamma_i}\rangle}$. Thus, we obtain an inductive description of $\bbV P_\calR(w\cdot\Lambda)$. 
	
	\vspace{.3cm}
	In addition, restricting the functor $\bbV$ to $\calM_{\calR,X}\subset \catO_{\calR,X}$, we obtain the result below.
	The categorical structure of $\catO_{\calR,X}$ is determined by its subcategory $\calM_{\calR,X}$, which allows us to further understand the structure of $\catO_{\calR,X}$.
	\begin{proposition}\label{structuretheorem}
		The functor 
		$\bbV: \calM_{\calR,X} \rightarrow \calZ_{\calR,X}\text{-}\mathrm{mod}$ is fully faithful, i.e., 
		\begin{align*}
			\xymatrixcolsep{3pc}\xymatrix{
				\Hom_{\calM_{\calR,X}}(M, N)  \ar[r]^-{\widesim{}} &  \Hom_{\calZ_{\calR,X}\text{-}\mathrm{mod}}(\bbV M, \bbV N) }
		\end{align*}
		is an isomorphism for any $M, N \in \calM_{\calR,X}$. 
	\end{proposition}
	\begin{proof}
		The proof follows the same strategy as the proof of Lemma \ref{isoms-of-funtors}. In the localization setting, Proposition \ref{Subgeneric-cases} provides the required identification, from which the statement follows.
	\end{proof}
	
	Let $U_q^\prime = U_q(\frakg^\prime)$ be the quantum group with respect to another finite dimensional  semisimple Lie algebra $\frakg^\prime$. In what follows, we temporarily use primed symbols $U_\bbF^\prime$, $\rmT_\bbF^\prime$, ${\Wch}^\prime$, $\catO_{q,\bbF}^\prime$, etc., to denote the counterparts of $U_\bbF$,  $\rmT_\bbF, \Wch, \catO_{q,\bbF}$ corresponding to the quantum group $U_q^\prime$.

	We now state the main result of this section. 
	
	\begin{theorem}\label{quan2quan}
		Let $\Lambda\in \rmT_\bbF$ and $\Lambda^\prime\in \rmT_\bbF^\prime$ be weights that are both either dominant or antidominant. Assume that there exists an isomorphism of Coxeter groups 
		\[(W_\Lambda, S_\Lambda)  \cong (W_{\Lambda^\prime}^\prime, S^\prime_{\Lambda^\prime}) \]
		which induces an isomorphism of groups $\Wch_\Lambda \rightarrow  \Wch_{\Lambda^\prime}^\prime$, $x \mapsto x^\prime$ and a bijection of linkage classes
		\[
		\xymatrix@C=1.5em{
			\overline{\Lambda} \ar[r]& \overline{\Lambda^\prime}, \hspace{-.6cm}& x\cdot\Lambda \ar@{|->}[r] & x^\prime\cdot\Lambda^\prime.
		}
		\]
		Then the corresponding blocks $\catO_{\bbF,\overline{\Lambda}}$  and $\catO_{\bbF,\overline{\Lambda^\prime}}^\prime$ are equivalent as categories. Under this equivalence, the Verma $U_\bbF$-module   $M_q(x\cdot\Lambda)$ and its simple quotient $L_q(x\cdot\Lambda)$ correspond to the Verma $U_\bbF^\prime$-module $M_q^\prime(x^\prime\cdot\Lambda^\prime)$ and its simple quotient $L_q^\prime(x^\prime\cdot\Lambda^\prime)$, respectively.
	\end{theorem}
	\begin{proof}
		By virtue of the equivalence \eqref{block2block}, we may fix the situation so that  $\Wch_{\Lambda}=W_\Lambda $ and  $\Wch_{\Lambda^\prime}^\prime=W_{\Lambda^\prime}^\prime$.
		To analyze the block  $\catO_{\bbF,\overline{\Lambda}} $ (or  $\catO_{\calR,\overline{\Lambda}} $), one only needs to understand the endomorphism algebra  $E_{\calR,\overline{\Lambda}}$.
		Indeed, the equivalence \eqref{Hom-functor}, together with the base change isomorphism  $P_\bbF(w\cdot\Lambda) \cong P_\calR(w\cdot\Lambda)\otimes_\calR \bbF $ and \eqref{Proj-basechange}, reduces the problem to describing the spaces $
		\Hom_{\catO_\calR}(P_\calR(x\cdot\Lambda), P_\calR(y\cdot\Lambda))$, for $x, y \in W_\Lambda$.
		Proposition~\ref{structuretheorem} gives a functorial identification involving  $\bbV$, so it remains to determine $\bbV P_\calR(x\cdot\Lambda)$. 
		The combinatorial description obtained earlier ensures that the structure of each $\bbV P_\calR(x\cdot\Lambda)$, and hence  of $ E_{\calR,\overline{\Lambda}}$, is governed entirely by the Coxeter system $(W_\Lambda, S_\Lambda)$ and the stabilizer of $\Lambda$. This completes the proof.
	\end{proof}

	The above theorem is a quantum analogue of Soergel’s structure description \cite{Soe90}.
	While we work with $\catO_{q,\bbF}$, its ``non-integral'' blocks can be described in a similar combinatorial manner by suitable extensions of integral Weyl subgroups and root subsystems.
	As indicated above, this generalization does not seem to introduce any essential difference.
	However, Example \ref{Notclosed} shows that integral root subsystems in the quantum setting exhibit certain deviations from the classical case, which will be discussed in the next section.

	\section{Minimal growth}\label{minGK}
	In this section, we determine the GK-dimension for each simple module in $\catO_{q,\bbF}$ via Lusztig's $\afun$-function. Furthermore, we compute the (nonzero)  minimal GK-dimension, highlighting a distinction from the classical case.
	
	\subsection{Lusztig's $\afun$-function} Let $S=\{s_\alpha \;|\; \alpha\in \Pi \}\subset W$. In \cite{Lus84}, Lusztig related the GK-dimension of simple objects in the \emph{principal block} $\catO_{\bar{0}}(\frakg)$ to a function $\afun: W \rightarrow \bbN$, defined for the Coxeter system $(W, S)$. We briefly recall the relevant definitions below. 
	
	Let $v$ be an indeterminate. A \emph{Hecke algebra} $\calH$ is a $\bbZ[v,v\inv]$-algebra generated by elements $T_w$, $w\in W$ subject to the following relations:
	\begin{align}
		&(T_s + v\inv)(T_s - v) = 0 \quad \text{for any } s \in S, \\
		& T_wT_{w'} = T_{ww'} \quad \text{if } l(ww') = l(w)+l(w'),
	\end{align}
	where $l$ is the \emph{length function} of $(W,S)$. 
	Let $\{C_w, w\in W \}$ be the \emph{Kazhdan-Lusztig basis} of $\calH$, which is characterized as the unique element $C_w\in \calH$ such that 
	\[\overline{C_w} = C_w, \quad C_w \equiv T_w (\mathrm{mod }\;  \calH_{<0}) \]
	where $\bar{{()}}: \calH \rightarrow \calH$ is the \emph{bar involution} such that $\bar{v} = v\inv$, $\overline{T_w } = T_{w\inv}\inv$ and $\calH_{<0} = \oplus_{w\in W}v\inv\bbZ[v\inv]T_w$. Assume that  $C_xC_y = \sum_{z\in W}h_{x,y,z}C_z$ for some $h_{x,y,z}\in \bbZ[v,v\inv]$. Then we have the following definition. 
	
	\begin{definition}
		The function $\afun: W \rightarrow \bbN$ is defined to be 
		\[\afun(z) = \max\{ \deg h_{x,y,z} \;|\; x,y\in W\} \]
		for each $z\in W$. 
	\end{definition}

	To each $\Lambda\in \rmT_\bbF$, we associate a Coxeter system $(W_\Lambda, S_\Lambda)$ (see Subsection \ref{IntegralWeyl&subsys}), from which we define the corresponding Hecke algebra $\calH_\Lambda$ and the Lusztig's function $\afun_\Lambda: W_\Lambda \rightarrow \bbN$. In general, $\calH_\Lambda$ need not be a subalgebra of $\calH$, and $\afun_\Lambda$ is not necessarily the restriction of $\afun: W \rightarrow \bbN$ since $S_\Lambda$ (any possible choice) cannot always be realized as a subset of $S$.

	Let $w_0$ be the \emph{longest element} of $(W,S)$. Note that $l(w_0) = |\Phi^+|$, $w_0\cdot 0 = -2\rho\in \frakh_{\bbQ}^*$ is antidominant and $L(ww_0\cdot 0)$, $w\in W$ are all simple objects in $\catO_{\bar{0}}(\frakg)$. By \cite{Lus84} the following property holds.
	\begin{lemma}\label{afun&GKdim}
		For each $w\in W$, we have 
		\[\pushQED{\qed}
		\afun(w) = l(w_0)-d_{U(\frakg)} \left(L(ww_0\cdot 0)\right).\qedhere 
		\popQED
		\]
	\end{lemma}

	Given $y, w \in W$, we write $y \leftarrow^L w$ if $C_y$ appears in the expansion of the product $CC_w$ for some $C\in \calH$ with respect to the basis $\{C_w, w\in W \}$. Extend $\leftarrow^L$ to a preorder $\prec^L$ on $W$. Let $\sim^L$ be the associated equivalence relation, i.e., $w \sim^L w^\prime$ iff $w\prec^L w^\prime$ and $w^\prime \prec^L w$. Moreover, we define $w\prec^R w^\prime$ by $w\inv \prec^L w^{\prime -1}$ and $w \sim^R w^\prime$ by $w\inv \sim^L w^{\prime -1}$. Let $\prec^{LR}$ be the preorder generated by $\prec^L$ and $\prec^R$. Similarly, we have an equivalence relation $\sim^{LR}$ on $W$ corresponding to $\prec^{LR}$. The equivalence classes on $W$ for $\sim^L$, $\sim^R$, $\sim^{LR}$ are called respectively \emph{left cells, right cells, two-sided cells} of $W$.  
	
	It is clear that $\{1 \}$ and $\{w_0 \}$ are two-sided cells of $W$. Let
	\[ \scrC=\{w\in W \;|\; w\neq 1 \text{ and  has a unique reduced expression}\}.  \]
	By \cite[Proposition 3.8]{Lus83} and \cite[Remark 3.3]{KL79} $\scrC$ and $w_0\scrC:=\{w_0w\;|\; w\in \scrC \}$ are two-sided cells of $W$. 

	\begin{lemma}\label{afunprop}
		Let $w, w^\prime \in W$. Then
		\begin{enumerate}[{\rm (1)}]
			\item $\afun(1)=0$, $\afun(w_0) = l(w_0)$.
			\item $\afun(w) = \afun(w\inv)$.
			\item If $w \prec^{LR} w^\prime$, then $\afun(w) \geqslant  \afun(w^\prime)$. Therefore, the function $\afun: W \rightarrow \bbN$ is constant on each two-sided cell of $W$. 
			\item For any $w \in  W\setminus\{1, w_0\}$, we have $\afun(\scrC) \leqslant \afun(w) \leqslant \afun(w_0\scrC)$. 
			\item $\afun(\scrC)=1$ and $\afun(w_0\scrC) =  l(w_0)-\inprod{\rho}{\theta_s}$, where $\theta_s$ is the short highest root in $\Phi$. 
		\end{enumerate}
	\end{lemma}
	\begin{proof}
		Assertions (1)-(3) appear in \cite{Lus03}. For (4), it suffices to prove that $x \prec^{LR} w \prec^{LR} y$ for any $x\in w_0\scrC$ and $y\in \scrC$. Since $w\neq 1$, there exists $s\in S$ with $l(w) = l(sw)+l(s)$; hence $w\prec^{R} s$ by \cite[(2.3a, 2.3b)]{KL79}. As $s\in \scrC$ and $\scrC$ is a two-sided cell of $W$, we obtain $w\prec^{LR} y$ for all $y\in \scrC$. The argument for $x\prec^{LR} w$ with $x \in w_0\scrC$ is analogous. For (5), the equality $\afun(\scrC)=1$ follows from $s \in \scrC$ and $\afun(s)=1$. By Lemma \ref{afun&GKdim}, the fact that the function $\afun$ attains its maximum value (other than $l(w_0)$) implies that
		$d\big(L(ww_0\cdot 0)\big)$
		reaches its minimal possible nonzero value, namely one half of the dimension of the \emph{minimal special} nilpotent orbit, which is precisely 
		$\inprod{\rho}{\theta_s}= h-1$, where $h$ is the Coxeter number of $\frakg$.
	\end{proof}

	\subsection{GK-dimension formula}  
	Let $\Lambda\in \rmT_\bbF$. Recall that  $W_\Lambda =\langle s_\alpha \;|\; \alpha\in \Phi_\Lambda \rangle$, and that $W_\Lambda = W$ when $\Lambda\in \calP$. 
	\begin{definition}\label{Isomorp}
		Define 
		$\check{(\text{-})}: W_\Lambda \rightarrow \Wch_\Lambda $ 
		to be the inverse map of the projection $\Wch_\Lambda \rightarrow W_\Lambda$. 
	\end{definition}
	Then we have the following result. 
	\begin{lemma}\label{GK-integral}
		Suppose that $\Lambda \in \calP$. Then $d(L_q(\Lambda)) =|\Phi^+| - \afun(w)$, where $w$ is the unique element in $W$ of minimal length such that $\check{w}\inv\cdot\Lambda$ is antidominant.
	\end{lemma}
	\begin{proof}
		By the equivalence (\ref{block2block}) and Lemma \ref{GKdimgeneralprop}(3), we may assume $\Lambda = q^\lambda$ for some $\lambda\in P$, so that $W_\Lambda = W$. By \cite[Theorem 4.2]{EK08} (or \cite[Theorem 6.2]{AM12}), we have $L_q(\Lambda)$ has the same character formula as the irreducible $\frakg$-module $L(\lambda)$. Using the canonical filtrations for $U_\bbF^-$ and $U(\frakn^-)$, we obtain $d_{U_\bbF} (L_q(\Lambda)) =d_{U(\frakg)}(L(\lambda))$. Let $w$ be the unique element of minimal length such that $w\inv\cdot\lambda$ is antidominant. Then, by \cite[Section 7.7]{Hum08}, we have 
		\begin{equation}\label{trans}
			T_{ww_0\cdot 0}^\lambda L(ww_0\cdot 0) = L(\lambda),
		\end{equation}
		where $T_\mu^{\mu^\prime}: \catO_{\bar{\mu}}(\frakg) \rightarrow \catO_{\overline{\mu^\prime}}(\frakg)$ is the Jantzen's translation functor defined by the composition of the natural projection $\catO(\frakg) \rightarrow  \catO_{\bar{\mu^\prime}}(\frakg)$ with the functor $L(\tilde{\nu})\otimes (\text{-})$, where $\tilde{\nu}\in W\nu \cap P^+$ and $\nu := \mu^\prime - \mu\in P$. By the biadjointness of $T_\mu^{\mu^\prime}$ (see \cite[Section 7.2]{Hum08}), $L(\lambda)$ and $L(ww_0\cdot 0)$ have the same Gelfand-Kirillov dimension. Hence, we have $d_{U_\bbF} \big(L_q(\Lambda)\big) = |\Phi^+| - \afun(w)$ by Lemma \ref{afun&GKdim}. 
	\end{proof}

	For a general weight $\Lambda$, define $\frakg_\Lambda$ to be the semisimple Lie algebra whose root system is $\Phi_\Lambda$, and set $U_q^\prime = U_q(\frakg_\Lambda)$. Note that $U_q^\prime$ is not necessarily a subalgebra of $U_q$ since $\Phi_\Lambda$ may not be closed in $\Phi$. The subalgebras $U_q^{\prime 0}$ and $U_q^{\prime\pm}$ are defined in the same way as for $U_q(\frakg)$. Since $\Phi_\Lambda \subset \Phi$ is a root  subsystem, we may regard $U_q^{\prime0}$ as a subalgebra of $U_q^0$ via the natural embedding $\bbZ\Phi_\Lambda \subset \bbZ\Phi$. Let $\Lambda^{\mathrm{res}}$ denote the restriction of $\Lambda$ to $U_\bbF^{\prime0}$. Then  $\Lambda^{\mathrm{res}}$ is clearly an integral weight of $U_\bbF^{\prime}$, and we have  $(W_\Lambda)_{\Lambda^{\mathrm{res}}}= W_\Lambda$. Moreover, $\mathrm{Stab}_{\Wch_\Lambda}(\Lambda) = \mathrm{Stab}_{\Wch_\Lambda}(\Lambda^{\mathrm{res}})$. By Theorem \ref{quan2quan} we have $\ch L_q(\Lambda) = \sum_{y\in W_\Lambda}a(y, \Lambda)\ch M_q(\check{y}\cdot\Lambda)$ for some integers $a(y,\Lambda)$ if and only if $\ch L_q(\Lambda^{\mathrm{res}}) = \sum_{y\in W_\Lambda}a(y, \Lambda)\ch M_q(\check{y}\cdot\Lambda^{\mathrm{res}}) $. On the other hand, for any $y\in W_\Lambda$ we have 
	\[d_{U_\bbF}(M_q(\check{y}\cdot \Lambda)) = |\Phi^+\setminus\Phi_\Lambda^+| + d_{U_\bbF^\prime}\left(M_q(\check{y}\cdot\Lambda^{\mathrm{res}})\right), \]
	where $\Phi^+_\Lambda := \Phi_\Lambda \cap \Phi^+$ (cf. Subsection \ref{Polynomialgrowth}). 
	So it implies that $d_{U_\bbF}(L_q( \Lambda)) = |\Phi^+\setminus\Phi_\Lambda^+| + d_{U_\bbF^\prime}(L_q(\Lambda^{\mathrm{res}}))$. Note that $\check{y}\cdot\Lambda$ is (anti-)dominant if and only if  $\check{y}\cdot\Lambda^{\mathrm{res}}$ is (anti-)dominant. Then Lemma \ref{GK-integral} yields that $d_{U_\bbF^\prime}(L_q(\Lambda^{\mathrm{res}})) = |\Phi_\Lambda^+| - \afun_\Lambda(w)$ where $w\in W_\Lambda$ is the unique element of minimal length such that $\check{w}\inv\cdot \Lambda$ is antidominant in the orbit $\bar{\Lambda}$. Combining the two equalities above, we get a dimension formula for $d\big(L_q(\Lambda)\big)$ as desired.

	\begin{theorem}\label{GKdimension}
		Let $\Lambda\in\rmT_\bbF$. Then we have 
		\[d(L_q(\Lambda)) = |\Phi^+| - \afun_\Lambda(w) \]
		where $w\in W_\Lambda$ is the unique element of minimal length such that $\check{w}\inv\cdot\Lambda$ is antidominant in the orbit $\bar{\Lambda}$. \qed
	\end{theorem}
	
	Every simple $U_\bbF$-module $L_q(\Lambda)$ for $\Lambda\in \calP^+$ is finite dimensional; hence its GK-dimension is zero.  To calculate the minimal GK-dimension of $L_q(\Lambda)$ for $\Lambda\in \rmT_\bbF\setminus\calP^+$. Theorem \ref{GKdimension} tells us the only thing we need do is to find the possible maximal value of the function $\afun_\Lambda$. We shall investigate it in the remainder of this section.
	
	\subsection{Maximal root subsystems} Let $\Phi$ be irreducible. It is well-known that the classification of all root subsystems of $\Phi$ may be deduced from that of closed root subsystems, see \cite{Car72}. The closed root subsystems were determined up to isomorphism by Borel-de Siebenthal \cite{BdS49} and also Dynkin \cite{Dyn52} by using the connection between Weyl groups and Lie algebras. A treatment of these results directly in terms of root systems may be found in \cite{Kan01} and \cite{DL11}. 
	\vspace{.3cm}
	
	Let us review some necessary facts.

	\begin{definition}
		We say that a proper (resp. closed) subsystem  $\Psi$ of $\Phi$ \emph{maximal} if there is no (resp. closed) subsystem $\Psi^\prime$ satisfying $\Psi \subsetneq \Psi^\prime \subsetneq \Phi$. 
	\end{definition}
	
	The following result is due to Borel-de Siebenthal \cite{BdS49}. 
	
	\begin{theorem}\label{Thm-max-closed}
		Expand the highest root $\theta$ in $\Phi$ with respect to the simple roots:
		\[\theta = \sum_{i=1}^{n}h_i\alpha_i. \]
		Then the maximal closed root subsystems of $\Phi$ (up to $W$-conjugacy) are those with bases (we use the hat ``$\hat{\alpha}$'' to denote elimination of $\alpha$): 
		\begin{enumerate}[{\rm (1)}]
			\item  $\{\alpha_1, \alpha_2, \dots, \hat{\alpha}_i,\dots, \alpha_n \}$ and $h_i =1$;
			\item $\{-\theta, \alpha_1,\dots, \hat{\alpha}_i,\dots, \alpha_n\}$ and $h_i $ is prime. \qed 
		\end{enumerate} 
	\end{theorem}

	We remark that any maximal root subsystem which is closed is necessarily a maximal closed subsystem; however, a maximal closed subsystem need not be maximal as a root subsystem.	In fact, the following result makes this precise, and one proof can be found in \cite[Corollary 2]{DL11}.

	\begin{lemma}\label{Lem-maxproper}
		If $\Psi$ is a maximal root subsystem of $\Phi$, then either $\Psi$ is closed in $\Phi$ or $\Psi^\vee$ is closed in $\Phi^\vee$.  \qed 
	\end{lemma}
	
	By combining Theorem \ref{Thm-max-closed} and Lemma \ref{Lem-maxproper}, it is easy to determine all maximal root subsystems. In particular,  Table \ref{Tablemaxrs}\footnote{We write $X_n^L$ (resp. $X_n^S$) for the isomorphism class of root subsystems of $\Phi$ formed from long (resp. short) roots in $\Phi$ with type $X_n$ in Table \ref{Tablemaxrs}.} below contains  all maximal root subsystems up to $W$-conjugacy in the cases of rank $n$ and of rank $n-1$. One can also find the result in \cite[Remark 8.4]{Osh}.
	
	{ \small \begin{table}[ht]
			\centering
			\begin{tabular}{c|l|l|l}
				\toprule
				\textsc{Type}& \textsc{Rank} $n-1$ & \textsc{Rank} $n$ & $\Psi_{max}$  \\
				\hline
				$A_n $	
				& 	
				$A_i\times A_{n-i-1}\; (0 \leqslant i \leqslant n-1)$
				& &$A_{n-1}$
				\\ 
				\hline
				$B_n $	
				& 	
				& $B_i\times B_{n-i}\; (1\leqslant i \leqslant n-1)$, ${D_n^L}$ & ${D_n^L}$
				
				\\
				\hline
				$C_n $ 
				
				& 
				
				&$C_i\times C_{n-i}\; (1\leqslant i \leqslant n-1)$, $D_n^S$ &$D_n^S$ \\ 
				\hline
				$D_n $	
				& 
				$A_{n-1}$, $D_{n-1}$
				& $D_i\times D_{n-i}\; (2 \leqslant i \leqslant n-2)$ & $D_{n-1}$ \\ 
				\hline
				$E_6$	
				& 
				$D_5$
				& $A_1\times A_5$, $A_2\times A_2\times A_2$ &$D_5$ \\ 
				\hline
				$E_7$	
				& 	
				$E_6$
				& $A_1\times D_6$, $A_7$, $A_2\times A_5$ & $E_6$\\
				\hline
				$E_8$	
				
				& 
				& $\begin{aligned}
					&D_8, A_1\times E_7, A_8 \\
					&A_2\times E_6, A_4\times A_4
				\end{aligned}$ & $A_1\times E_7$\\
				\hline
				$F_4$	
				& 
				
				& 
				$B_4$, $C_4$, $A_2^L\times A_2^S$ & $B_4$, $C_4$
				\\ 
				\hline
				$G_2$	
				& 
				
				& $A_1^L\times A_1^S$, $A_2^L$, $A_2^S$& $A_2^L$, $A_2^S$ \\ 
				\bottomrule
			\end{tabular}
			\vspace{.1cm}
			\caption{Maximal root subsystems of $\Phi$}
			\label{Tablemaxrs}
	\end{table}}
	
	\vspace{.3cm}	
	We associate to $\Phi$ an integer $\gamma = \gamma(\Phi)$, defined as the least common multiple of all prime coefficients $h_i$ in the expression of the highest root $\theta$ of $\Phi$. Clearly, $\gamma(\Phi) = \gamma(\Phi^\vee)$. By Theorem \ref{Thm-max-closed}, $\gamma$ is the minimal positive integer such that $\gamma Q \subset \bbZ \Psi$ for any maximal closed root subsystem $\Psi\subset \Phi$ of full rank. 
	
	The following result provides a characterization of maximal root subsystems. 
	\begin{proposition}\label{maxroot-characterizations}
		Let $\bbF = \bbC(q^{\frac{1}{\gamma}})$. Then any maximal root subsystem $\Psi$ of $\Phi$ is of the form $\Phi_\Lambda$ for some $\Lambda\in \rmT_\bbF$. In particular, if $\Psi^\vee$ is closed in $\Phi^\vee$, then $\Psi = \Phi_\lambda$ for some $\lambda\in \frac{1}{\gamma}P$. 
	\end{proposition}
	\begin{proof}
		Suppose that $\Psi$ is a maximal root subsystem of $\Phi$. By Lemma \ref{Lem-maxproper}, either $\Psi$ is closed in $\Phi$, or $\Psi^\vee$ is closed in $\Phi^\vee$. For the latter case, it suffices to find $\lambda\in \frac{1}{\gamma}P$ such that 
		\begin{equation}\label{inclusions}
			\Psi \subset \Phi_\lambda\subsetneq \Phi.
		\end{equation}
		Since $\Psi^\vee$ is closed in $\Phi^\vee$, the proposition in \cite[Section 12-3]{Kan01} gives a proper inclusion of $\bbZ$-lattices 
		\[Q^\prime :=\bbZ\Psi^\vee \subsetneq \bbZ\Phi^\vee = Q^\vee. \]
		If $\mathrm{rank}\;Q^\prime = \mathrm{rank}\;Q^\vee$, then by Theorem \ref{Thm-max-closed}, the quotient $Q^\vee/Q^\prime$ is torsion and $\gamma Q^\vee \subset Q^\prime$. Their corresponding dual lattices satisfy that 
		$\Hom_\bbZ(Q^\vee, \bbZ) \subsetneq \Hom_\bbZ(Q^\prime, \bbZ).$
		Since the bilinear form $(\cdot,\cdot)$ is nondegenerate on $\frakh_\bbQ^*$, each homomorphism in $\Hom_\bbZ(Q^\prime, \bbZ)$ corresponds uniquely to an element $\lambda$ in $\frac{1}{\gamma}P$, while  $\Hom_\bbZ(Q^\vee, \bbZ)$ identifies with $P$ under this form, that is, 
		\[ 	P=\Hom_\bbZ(Q^\vee, \bbZ) \subsetneq \Hom_\bbZ(Q^\prime, \bbZ) \subset \frac{1}{\gamma}P.  \]
		Such $\lambda$ is exactly the desired weight satisfying ~\eqref{inclusions} as long as $\lambda\notin P$. If $\mathrm{rank}\;Q^\prime < \mathrm{rank}\;Q^\vee$, then the orthogonal complement 
		\[(Q^\prime)^\perp = \{\lambda\in \frakh_\bbQ^* \mid (\lambda, \mu) = 0, \forall \mu \in Q^\prime \} \]
		is nonzero, and any $\lambda\in (Q^\prime)^\perp\cap (\frac{1}{\gamma}P\setminus P) $  satisfies ~\eqref{inclusions}. Then the assertion for this case is proved by Proposition \ref{int-root-sys}.
		For the former case, we may assume $\mathrm{rank}\;\Psi =n$ since any maximal root subsystem of rank $n-1$ has to be dual-closed (see Table \ref{Tablemaxrs}). Using the fact that $w\Phi_\Lambda = \Phi_{w\Lambda}$ for any $w\in W$ and $\Lambda\in \rmT_\bbF$, it suffices to show that any root subsystem $\Psi$ of type (2) in Theorem \ref{Thm-max-closed} is of form $\Phi_\Lambda$ for some $\Lambda\in \rmT_\bbF$. Assume that $\Psi$ has a base 
		\[\{-\theta, \alpha_1, \cdots, \hat{\alpha}_i, \cdots, \alpha_n \}, \]
		where $h_i$ is prime, and fix a primitive $2h_i$-th root of unity $\varepsilon$. Choose one $\Lambda$ in $ \rmT_\bbF$ (also in $\rmT_q$) as follows:
		\[\Lambda(K_i) = \varepsilon q^{d_i}, \quad \Lambda(K_j) = 1, \quad \text{for all } j\neq i.\]
		One checks that $2h_id_i$ is always divisible by $(\theta, \theta)$, and since $\Psi$ is closed, $\Psi = \bbZ\Psi \cap \Phi$. Hence $\Psi = \Phi_\Lambda$. 
	\end{proof}
	
	From the above proof, every maximal root subsystem of $\Phi$ can be realized as $\Phi_\Lambda$ for some $\Lambda\in \rmT_\bfk$. On the other hand, in general a root subsystem of $\Phi$ need not be of the form $\Phi_\Lambda$, as noted in Remark \ref{Rmk-counterexample}.

	\subsection{Minimal GK-dimensions} 
	Let $\frakg$ be a simple Lie algebra throughout this subsection. Our aim is to determine the minimal GK-dimensions of irreducible objects in $\catO_{q,\bbF}$. We proceed by a comparative analysis of three cases, according to whether the highest weight lies in one of the following subsets of $\rmT_{\bbF}$: $\rmT_{\bbF} \setminus \calP$, the set of linear but non-integral weights, and $\calP \setminus \calP^+$.
	\vspace{.3cm}

	Let $\Psi_{max}$ denote a proper root subsystem of $\Phi$ whose cardinality $|\Psi_{max}|$ is maximal among all proper subsystems of $\Phi$. Then $\Psi_{max}$ is necessarily a maximal root subsystem. The possible types of $\Psi_{max}$ are listed in Table \ref{Tablemaxrs}.  Define $\Psi_{max}^+ = \Phi^+\cap \Psi_{max}$. Set $\kappa_0= |\Phi^+\setminus\Psi_{max}^+|$.  Then we have
	
	\begin{proposition}\label{nonlinearGK}
		Let $\Lambda \in \rmT_\bbF\setminus \calP$.    Then  $d\big(L_q(\Lambda)\big) = \kappa_0$
		if and only if 
		\begin{enumerate}[{\rm (1)}]
			\item $\Lambda$ is dominant regular with respect to the equivalence class $\sim$, and
			\item the associated proper subsystem $\Phi_\Lambda$ has maximal cardinality among all proper subsystems. 
		\end{enumerate}
		
	\end{proposition}
	\begin{proof}
		Let $w\in W_\Lambda$ be the unique element of minimal length such that $\check{w}\cdot \Lambda$ is antidominant in  $\overline{\Lambda}$. If $\Lambda$ is dominant regular, then $w$ is the longest element in $W_\Lambda$. In this case, $\afun_\Lambda(w)$ equals $|\Phi_\Lambda^+|$ by Lemma \ref{afunprop}, which proves the ``if'' part. For the ``only if'' part, since $\Lambda\notin \calP$, $\Phi_\Lambda$ is proper in $\Phi$, and hence we have 
		\[2\afun_\Lambda(w) \leqslant |\Phi_\Lambda| \leqslant |\Phi|-2\kappa_0. \]
		If $d(L_q(\Lambda)) = \kappa_0$, then $\afun_\Lambda(w) = |\Phi^+| - \kappa_0$ and $|\Phi_\Lambda| =|\Phi|-2\kappa_0$. Thus $\Phi_\Lambda$ is a proper subsystem of $\Phi$ of maximal cardinality, and $w$ must be longest (thus unique) in the Coxeter system $(W_\Lambda,S_\Lambda)$. By Lemma \ref{uniq2reg}, it follows that $\Lambda$ is dominant regular. 
	\end{proof}

	For a field $\bbF$ with $\bbQ(q)\subset \bbF \subset \bfk$, it is not always true that there exists $\Lambda \in \rmT_\bbF \setminus \calP$ such that the associated proper subsystem $\Phi_\Lambda$ coincides with some  $\Psi_{max}$ of a prescribed type. For instance, consider $\bbF = \bbQ(q^{\frac{1}{\gamma}})$. In this situation, no  $\Phi_\Lambda$ for $\Lambda\in \rmT_{\bbF}$ can be realized in the following cases:
	\[(\Phi, \Psi_{max}) : \quad (B_n, D_n^L), \quad (F_4, B_4), \quad (G_2, A_2^L).\]
	To illustrate, we take the case $(B_n, D_n^L)$ as an example; the other cases follow similarly. 
	\begin{example}\label{TypeB-Const}
		Let $\Phi$ be a root system of type $B_n$. To be consistent with our convention, let $\Phi$ consist of the vectors $\pm \epsilon_i$ (of squared length $2$) and $\pm (\epsilon_i\pm \epsilon_j)$ (of squared length $4$), $1\leqslant i\neq j\leqslant n$, where the $\epsilon_i'$s form an orthogonal basis of $\frakh^*$ with $(\epsilon_i, \epsilon_i)=2$. Take $\Psi_{max}=\{\pm(\epsilon_i\pm\epsilon_j)\mid 1\leqslant i\neq j\leqslant n \}$, which is of type $D_n$ with a base $\{\epsilon_1-\epsilon_2, \cdots, \epsilon_{n-1}-\epsilon_n, \epsilon_{n-1}+\epsilon_n \}$, and consider $\bbF = \bbQ(q^{\frac{1}{\gamma}})$. If $  \Psi_{max}=\Phi_\Lambda$ for some $\Lambda\in \rmT_{\bbF}$, then we would have 
		\[\Lambda(K_{2(\epsilon_{n-1}-\epsilon_n)}) \in q^{4\bbZ} \quad \text{and}\quad  \Lambda(K_{2(\epsilon_{n-1}+\epsilon_n)})\in q^{4\bbZ}, \quad \text{but}\quad \Lambda(K_{2\epsilon_n})\notin q^{2\bbZ}.  \]
		This forces the value $\Lambda(K_{\epsilon_n})$ to be of the form $\imath q^m$ or $-\imath q^m$ for some $m\in \bbZ$, contradicting the assumption that $\Lambda\in \rmT_{\bbF}$. \qed 
	\end{example}
	
	When $\Lambda$ ranges over all linear (or integral) weights, the minimal GK-dimensions are already well established, as they can be deduced from the corresponding classical results via the classical-quantum equivalence (see \cite[Theorem 4.2]{EK08}). Now consider the following subsets of $\rmT_\bfk$: 
	\[ X_1=\bbZ_2^n\rtimes\{q^\lambda\mid \lambda\in \frac{1}{\gamma}P \}\setminus\calP^+,\qquad X_2 = \calP\setminus\calP^+,  \]
	and define
	\[
	\kappa_i= \min\{\, d(L_q(\Lambda)) \mid \Lambda \in X_i \,\}, \quad \text{for } i=1, 2.
	\]
	Since $X_2\subset X_1$, it follows that $\kappa_1\leqslant\kappa_2$. 
	Moreover, we have the following result. 
	\begin{lemma}\label{linearGKdim}
		Let $\theta_s$ denote the highest short root in $\Phi$ (which coincides with the highest root $\theta$ in the simply-laced types). Then 
		\[\kappa_1= \inprod{\rho}{\theta}, \quad \kappa_2 = \inprod{\rho}{\theta_s}. \]
	\end{lemma}
	\begin{proof}
		This is a direct proof from Lemma  \ref{afunprop}(5) and Proposition \ref{maxroot-characterizations}. 
	\end{proof}
	
	\begin{remark}
		The classical results about minimal GK-dimension can be found in \cite[Table 1]{Jos74} with corrections for  $E_{7,8}$ in \cite[Section3]{Jos76} and \cite[Table 1]{BMXX}.
	\end{remark}
	
	Let $h$ and $h^\vee$ denote the \emph{Coxeter number} and \emph{dual Coxeter number} of $\frakg$,  respectively (cf. \cite[Chapter 6]{Kac90}). By Lemma \ref{linearGKdim}, we have $\kappa_1 = h^\vee - 1$ (see also \cite{Wan99}) and $\kappa_2 = h-1$. Combining Proposition \ref{maxroot-characterizations}, Proposition \ref{nonlinearGK} and Lemma \ref{linearGKdim} we obtain the following result. 
	\begin{theorem}\label{q-minGK}
		For a field $\bbF $ with $ \bbC(q^{\frac{1}{\gamma}})\subset \bbF \subset \bfk$, we have 
		\[ 	\min\{\, d\left(L_q(\Lambda)\right) \mid \Lambda \in \rmT_{\bbF}\setminus\calP^+  \,\} = \min\{\kappa_0, \kappa_1 \}.  \]
		The explicit values of $\kappa_i$ are listed in Table \ref{MinGK}, where $\min\{\kappa_0, \kappa_1\}$ is highlighted in red for each type. \qed 
		
		{ \small \renewcommand{\arraystretch}{1.2}%
			\begin{table}[h]
				\centering
				\begin{tabular}{cccccccccc}
					\toprule
					$\mathsf{Type}$ & $A_n $	& 		$B_n $	& 	$C_n $ & 	$D_n $	& 	$E_6$	&  	$E_7$	&	$E_8$	&	$F_4$	& 	$G_2$	\\
					\hline
					
					$\kappa_0$ & \cellcolor{red!25}{$n$} &\cellcolor{red!25}{$n$}&\cellcolor{red!25}{ $n$} &$2n-2$ &$16$ &$27 $& $56$ &\cellcolor{red!25}$8$ &\cellcolor{red!25}$3$ \\ $\kappa_1$ & \cellcolor{red!25}{$n$} &$2n-2$& \cellcolor{red!25}{$n$} &\cellcolor{red!25}{$2n-3$} &\cellcolor{red!25}{$11$} &\cellcolor{red!25}$17 $&\cellcolor{red!25} $29$ &\cellcolor{red!25}$8$ &\cellcolor{red!25}$3$\\
					$\kappa_2$ & \cellcolor{red!25} $n$ &$2n-1$& $2n-1$ &\cellcolor{red!25}$2n-3$ &\cellcolor{red!25}$11$ &\cellcolor{red!25}$17 $&\cellcolor{red!25} $29$ &$11$ &$5$  \\
					\bottomrule
				\end{tabular}
				\vspace{.1cm}
				\caption{Minimal GK-dimensions over $\bbC(q^{\frac{1}{\gamma}})$}
				\label{MinGK}
		\end{table}}
	\end{theorem}

	From Table \ref{MinGK}, we see that, except for  type $ B_n (n\geqslant 3)$, the minimal values coincide with those in the classical case, i.e., $\kappa_1$. By the previous statement (see Example \ref{TypeB-Const}), the minimal GK-dimension for each type over the field $\bbQ(q^{\frac{1}{\gamma}})$ is also $\kappa_1$.

	\section{Application}\label{App}
	In this section, we apply the results in the previous section to address the existence problem for quantum cuspidal weight  modules, which is central to the classification of simple weight modules via parabolic induction.

	Recall that a simple weight $\frakg$-module $L$ (with $\dim L_\lambda < \infty$ for any $\lambda \in \mathfrak{h}^*$) which cannot be induced from a  module over any proper parabolic subalgebra of $\frakg$ is called \emph{cuspidal}. It is known that such modules exist only when $\mathfrak{g}$ consists of simple components of type $A$ or type $C$.
	Further properties of cuspidal modules can be found in \cite{Fer90}. In brief, they are infinite-dimensional, their weight spaces have uniformly bounded dimensions, and every root vector of $\frakg$ acts injectively on them (consequently, all the weight spaces have the same dimension).
	Let $\check{U}_q := U_q\otimes_{\bbC}\bfk$. Now let $M_q$ be an infinite-dimensional irreducible weight $\check{U}_q$-module. Suppose that there exists a positive integer $d$ such that
	\[ \dim_\bfk (M_q)_\Lambda = d, \quad \text{for all } \Lambda \in \Omega(M_q). \]
	Clearly, $M_q$ is finitely generated, and any nonzero weight space can serve as a generating space of $M_q$. For each $\beta \in Q$, let $(\check{U}_q)_\beta$ be the $\beta$-weight space of $\check{U}_q$ with respect to the adjoint action of the Cartan part $\check{U}_q^0$. For any monomial $u = F^{\underline{k}}K_\mu E^{\underline{r}}\in \check{U}_q$, define the height of this monomial by 
	\[\mathrm{ht}(u):= \sum_i(k_i+r_i)\mathrm{ht}(\beta_i), \]
	where $\mathrm{ht}(\beta_i)$ is the height of positive root $\beta_i$ with respect to the simple roots in $\Pi$, and denote by $\check{U}_q^{(m)}$ the $\bfk$-linear span of the monomials $u\in \check{U}_q$ such that $\mathrm{ht}(u)\leqslant m$.  Then we get a filtration $\{\check{U}_q^{(m)} \}_{m\geqslant 0}$ for  $\check{U}_q$, and $\check{U}_q^{(1)}$ is a generating space of $\check{U}_q$. 
	Fix one nonzero weight space $(M_q)_{\Lambda_0}$ of $M_q$. Set $r = \max\{(\rho, \alpha^\vee)\mid \alpha\in \Phi \}$. Then we have 
	\[\sum_{\substack{|(\rho, \beta)|\leq m \\ \beta\in Q}}(\check{U}_q)_\beta (M_q)_{\Lambda_0} \subset \check{U}_q^{(m)}(M_q)_{\Lambda_0}\subset \sum_{\substack{|(\rho, \beta)|\leq mr \\ \beta\in Q}}(\check{U}_q)_\beta (M_q)_{\Lambda_0}, \]
	for all $m\in \bbN$. Hence, 
	\[d(2m+1)^n \leqslant \dim_{\bfk} \check{U}_q^{(m)}(M_q)_{\Lambda_0} \leqslant d(2mr+1)^n. \]
	This implies that $\lim_{m\rightarrow\infty}\log_m \dim_\bfk \check{U}_q^{(m)} (M_q)_{\Lambda_0} = n$ and thus, the GK-dimension $d(M_q)$ is exactly the rank of $\frakg$. Moreover, we have the following result. 
	
	\begin{theorem}\label{Existence}
		The algebra $\check{U}_q$ admits cuspidal modules if and only if  the underlying semisimple Lie algebra $\frakg$ consists of simple components of type $A$, $B$ or $C$. 
	\end{theorem} 
	\begin{proof}
		Without loss of generality, we may assume that $\frakg$ is simple. Let $M_q$ be such a $\check{U}_q$-module.  Consider its annihilator $\mathrm{Ann}_{\check{U}_q}M_q$, which is a primitive ideal of $\check{U}_q$. By \cite[Section 6]{JL95}, this ideal also annihilates a certain highest weight module, say $L_q(\Lambda)$ for some $\Lambda\in \rmT_\bfk$. Since the representation map  $\check{U}_q/\mathrm{Ann}_{\check{U}_q}M_q \rightarrow \End(M_q)$ is injective, and $\check{U}_q$ is Noetherian, it follows that $d\big(\check{U}_q/\mathrm{Ann}_{\check{U}_q}M_q\big) \leqslant 2d(M_q) = 2n $. Thus, by \cite[Lemma 6.2 (v)]{JL95} and the fact that $L_q(\Lambda)$ has the same annihilator $\mathrm{Ann}_{\check{U}_q}M_q$ we have $d(L_q(\Lambda))=n$. From our previous results, such Gelfand-Kirillov dimension $n$ occurs only when $\frakg$ is of type $A$, $B$ or $C$. Conversely, by \cite{FHW} and \cite[Theorem 7.2]{CGLW21}, one can explicitly construct  a family of such modules with each weight multiplicity $1$ for quantum groups of type $A$, $B$ and $C$. Thus, the theorem follows. 
	\end{proof}
	
	\begin{remark}
    \begin{enumerate}[{\rm (i)}]
        \item The result in Theorem \ref{Existence} holds not only for an indeterminate q but also for any specialization $q \to \xi$, where $\xi \in \mathbb{C}$ is transcendental. In fact, since the construction in \cite{CGLW21} applies to every non–root of unity $\xi$, it is natural to expect that the conclusion of Theorem \ref{Existence} should remain valid for all specializations $q \to \xi$ with $\xi$ not a root of unity. 
        \item In fact, the above discussion, together with Example \ref{TypeB-Const} and Theorem \ref{q-minGK}, shows that the existence of cuspidal modules for $U_q$ leads to the same conclusion as in Theorem \ref{Existence}. In contrast, when considering $U_{\bbQ(q)}$, cuspidal modules of type $B$ do not exist.
   \end{enumerate}
	\end{remark}

	\appendix
	\section{Translation functors and structure functors}\label{Appendix}
	The translation and structure functors for $\catO_{q,\bbF}$ behave essentially as in the classical case. However, in our setting, special care should be taken for the ``non-integral'' blocks. In this appendix, we summarize their properties in the deformed framework, mainly following \cite{AJS94} and \cite{Fie06}.
	\subsection{Translation functors}\label{App-Tranfun}
	Let $R$ be a deformation algebra. For any module $ M \in \catO_R $ and any finite-dimensional module $ E \in \catO_{q,\bbF} $,  the tensor product $ E \otimes M $, where we write $\otimes $ for $ \otimes_\bbF$, carries a natural $U_R$-module structure making it an object of $\catO_R $; specifically, the action of $R$ is defined on the second factor only, while the $U_\bbF$-action is given via the comultiplication $\Delta$. These two actions are clearly compatible. The weight space of $E\otimes M$ is of form 
	\[ (E\otimes M)_{\tau\Lambda} = \oplus_{(\mu,\Lambda^\prime)}E_{q^\mu}\otimes  M_{\tau\Lambda^\prime},  \]
	where the sum is taken over all pairs $(\mu,\Lambda^\prime)\in P\times\rmT_\bbF$ such that $q^\mu\Lambda^\prime = \Lambda$. Moreover, one has in general a compatibility with base change: for any morphism $R\rightarrow R^\prime$ of deformation algebras, there is a natural isomorphism 
	\begin{equation}\label{base+tensorproduct}
		(E\otimes M)\otimes_RR^\prime \cong E\otimes (M\otimes_RR^\prime)
	\end{equation}
	in the category $\catO_{R^\prime}$. 
	\begin{lemma}\label{tensoringwithfdimreps}
		For any $\Lambda\in \rmT_\bbF$ and any finite dimensional module $E\in \catO_{q,\bbF}$, there exists a finite decreasing filtration
		\[E\otimes  M_R(\Lambda) = M_1 \supset M_2 \supset \cdots \supset M_r \supset 0 \]
		in $\catO_R$ with quotients isomorphic to $M_R(q^{\nu}\Lambda )$, where $q^\nu$ runs over all weights of $E$ counted with weight multiplicity.   
	\end{lemma}
	\begin{proof}
		Order a basis of $E$ consisting of vectors $v_1, \dots, v_r$ with weights $q^{\nu_1}, \dots, q^{\nu_r}$ (repetitions allowed) so that $i\leqslant j$ whenever $\nu_i\leq \nu_j$. This induces a filtration $0\subset E_r \subset \cdots \subset E_1=E$ where $E_i$ is the $U_q^0U_q^+$-submodule spanned by $v_i, \dots, v_r$. Note that $E \otimes M_R(\Lambda)$ is isomorphic to a free $U_R^-$-module, and each $E_i \otimes M_R(\Lambda)$ is a free $U_R^-$-submodule of rank $\dim E_i$ (with the trivial $U_R^-$-action on $E_i$). Hence, the corresponding $U_R$-module filtration on $E\otimes M_R(\Lambda)$ has  quotients isomorphic to the Verma $U_R$-modules $M_R(q^{\nu_i}\Lambda )$.
	\end{proof}
	
	For any $\Lambda,\Lambda^\prime\in \rmT_\bbF$ such that  $\Gamma:=\Lambda^\prime\Lambda\inv\in \calP$, there exists a unique element of $\calP^+$ lying in the orbit $\Wch_\Gamma\Gamma$; we denote it by $\overline{\Gamma}:=w\Gamma$ for some  $w\in \Wch_\Gamma$. Let $R$ be a local deformation algebra, and let $X, X^\prime\in \rmT_\bbF/\sim_R$ be the linkage classes containing $\Lambda$ and $\Lambda^\prime$, respectively. For any module $M\in \catO_{R,X}$, we define the \emph{translation functor }
	\[T_\Lambda^{\Lambda^\prime}(M) = \pr_{X^\prime}\left(L_q(\overline{\Gamma})\otimes M\right) \]
	where $\pr_X$ denotes the projection functor from $\catO_R$ onto the block $\calO_{R,X}$.
	\begin{proposition}\label{prop-T}
		Let $X, X^\prime\in \rmT_\bbF/\sim_R$ be the linkage classes containing $\Lambda$ and $\Lambda^\prime$, respectively. Suppose that  $\Lambda^\prime\Lambda\inv\in \calP$. Then 
		\begin{enumerate}[{\rm (1)}]
			\item $T_{\Lambda}^{\Lambda^\prime}$ is an exact functor from $\catO_{R, X}$ to $\catO_{R,X^\prime}$.
			\item For any morphism $R\rightarrow R^\prime$ of local deformation algebras, we have 
			\begin{equation*}
				T_\Lambda^{\Lambda^\prime}(M\otimes_RR^\prime) \cong T_\Lambda^{\Lambda^\prime}(M )\otimes_RR^\prime 
			\end{equation*}
			for any $M\in \catO_{R,X}$. 
			\item If $P\in \catO_{R,X}$ is projective, then $T_\Lambda^{\Lambda^\prime}(P)$ is also projective in $\catO_{R, X^\prime}$. 
			\item $T_{\Lambda}^{\Lambda^\prime}$ is both left and right adjoint to $T_{\Lambda^\prime}^{\Lambda}$, and vice versa.
			
		\end{enumerate}
	\end{proposition}
	\begin{proof}
		The first statement is clear, and the second is deduced from the isomorphism \eqref{base+tensorproduct}. Note that if $P\in \catO_{R}$ is projective, then $L\otimes P\in \catO_{R}$ is also  projective for any finite dimensional $U_q$-module $L$, and then (3) follows. For (4), interchanging the roles of $\Lambda$ and $\Lambda^\prime$, we obtain the translation functor $T_{\Lambda^\prime}^\Lambda$, where the involved finite-dimensional simple module is isomorphic to the dual $L_q(\overline{\Gamma})^*$. For any $M\in \catO_{R,X}$, $N\in \catO_{R,X^\prime}$, we have the following isomorphism 
		\begin{equation}\label{adj1}
			\Hom_{\catO_R}(M, T_{\Lambda^\prime}^\Lambda(N)) \cong \Hom_{\catO_R}(T_\Lambda^{\Lambda^\prime}(M),N), 
		\end{equation}
		and identifying the $U_\bbF$-module $L_q(\overline{\Gamma})$ with its double dual via the map $v \mapsto (v: f \mapsto f(K_{2\rho}\inv v))$, we obtain another isomorphism 
		\begin{equation}\label{adj2}
			\Hom_{\catO_R}(N, T_{\Lambda}^{\Lambda^\prime}(M)) \cong \Hom_{\catO_R}(T_{\Lambda^\prime}^{\Lambda}(N),M). 
		\end{equation}
		The isomorphisms~\eqref{adj1} and ~\eqref{adj2} show that $T_{\Lambda}^{\Lambda^\prime}$ is both left and right adjoint to $T_{\Lambda^\prime}^{\Lambda}$, and vice versa. For further details, see \cite[Subsection 7.6]{AJS94}. 
	\end{proof}
	We further assume that the integral weight $\Gamma$ is of the form $q^\nu$ for some $\nu\in P$. In this case, we say that $\Lambda$ and $\Lambda^\prime$ are compatible. Then Corollary \ref{int-Weyl} yields that 
	\[\Phi_{R,\Lambda} = \Phi_{R,\Lambda^\prime} \quad \text{and}\quad \Wch_{R,\Lambda} = \Wch_{R,\Lambda^\prime}. \]
	Recall the condition \eqref{WeylCha}: two compatible weights $\Lambda, \Lambda^\prime$ lie in {the closure of the same ``facet''} for $\Wch_{R,\Lambda}$ if they satisfy 
	\begin{equation}\label{WeylCham}
		q^{2(\rho, \alpha)}(\tau\Lambda)(K_{2\alpha})\in q^{\bbN(\alpha,\alpha)} \quad \text{iff}\quad q^{2(\rho, \alpha)}(\tau\Lambda^\prime)(K_{2\alpha})\in q^{\bbN(\alpha,\alpha)}\quad \text{for all } \alpha\in \Phi_{R,\Lambda}.
	\end{equation}
	Intuitively, the weights $\tau\Lambda$ and $\tau\Lambda^\prime$ satisfying the above condition determine two integral weights in the Euclidean subspace $E(\tau\Lambda)$ spanned by $\Phi_{R,\Lambda}$, which lie in the closure of the same facet in the classical sense. 
	
	\begin{lemma}\label{Tran-Verma}
		Let $\Lambda, \Lambda^\prime\in \rmT_\bbF$ be antidominant weights with respect to $\sim_R$. Suppose $\Lambda, \Lambda^\prime$ are  compatible and satisfy the condition \eqref{WeylCham}. Let $w\in \Wch_{R,\Lambda}$. Then $T_\Lambda^{\Lambda^\prime}M_R(w\cdot\Lambda)$ has a Verma flag with subquotients isomorphic to  
		$M_R(w\bar{x}\cdot\Lambda^\prime)  $
		for $ \bar{x}\in \Stab_R(\Lambda)/\Stab_R(\Lambda)\cap \Stab_R(\Lambda^\prime)$, each occurring once.
	\end{lemma}
	\begin{proof}
		First, $T_\Lambda^{\Lambda^\prime}(M_R(w\cdot\Lambda))$ admits a Verma flag whose subquotients are parametrized by  $q^\nu w\cdot\Lambda$ lying in $\Wch_{R,\Lambda^\prime}\cdot\Lambda^\prime$. Associate with $\tau\Lambda$ its integral part $\hat{\lambda}\in E(\tau\Lambda)$ such that $(\hat{\lambda}, \alpha^\vee)\in \bbZ$ and $(\tau\Lambda)(K_\alpha) = \sigma(K_\alpha)q^{(\hat{\lambda},\alpha)}$ for all $\alpha\in \Phi_{R,\Lambda}$, with $\sigma\in \bbZ_2^n$ determined by the signs of $(\tau\Lambda)(K_\alpha)$. Then it can be checked that 
		$(\tau\Lambda)(s_{R,\alpha}^\Lambda\cdot\tau\Lambda)\inv(K_\beta) = q^{(\hat{\lambda}-s_\alpha\cdot \hat{\lambda}, \beta)}$ for any $ \alpha, \beta\in \Phi_{R,\Lambda}$. Restricting $\check{(\text{-})}$ from Definition \ref{Isomorp} we get the stablizer of $\hat{\lambda}$ in $W_{R,\Lambda}$ with repsect to the shifted action is isomorphic to $\Stab_{\Wch_{R,\Lambda}}(\Lambda)$. 
		 Using \eqref{WeylCham}, we may work with the integral parts of $\tau\Lambda, \tau\Lambda^\prime$ inside the Euclidean space $E(\tau\Lambda)$. Repeating the classical argument of \cite[Section II. 7.8]{Jan03} and \cite[Lemma 7.5]{Hum08}, one verifies that the weights that occur are precisely $w\bar{x}\cdot \Lambda^\prime$ for $ \bar{x}\in \Stab_R(\Lambda)/\Stab_R(\Lambda)\cap \Stab_R(\Lambda^\prime)$, and that each Verma module appearing as a subquotient has multiplicity $1$. 
	\end{proof}
	
	\begin{remark}
		It follows that the translation functor $T_\Lambda^{\Lambda^\prime}$ sends modules with a Verma flag to modules with a Verma flag. 
	\end{remark}
	
	\begin{corollary}\label{App-tran-proj}
		Let 
		$X, X^\prime \in \rmT_\bbF/\sim_R$ be the linkage classes containing antidominant weights
		$\Lambda$ and $\Lambda^\prime$, respectively. 	Let $\Lambda, \Lambda^\prime$ be as in the above lemma, and let moreover  $\Stab_R(\Lambda) \subset \Stab_R(\Lambda^\prime)$. Then the following hold: 
		\begin{enumerate}[{\rm (1)}]
			\item For $w\in \Wch_{R,X}$, the module $T_\Lambda^{\Lambda^\prime}P_R(w\cdot\Lambda)$ decomposes into $|\Stab_R(\Lambda^\prime)/\Stab_R(\Lambda)|$ copies of $P_R({w}\cdot\Lambda^\prime)$,  possibly with additional summands $P_R({w}'\cdot\Lambda^\prime)$ with ${w}'\cdot\Lambda^\prime>{w}\cdot\Lambda^\prime$. 
			\item Suppose further that $\Stab_R(\Lambda) = \Stab_R(\Lambda^\prime)$. Then the functors $T_\Lambda^{\Lambda^\prime}: \catO_{R, X}\to \catO_{R, X^\prime}$ and $T_\Lambda^{\Lambda^\prime}:\catO_{R, X^\prime} \to \catO_{R, X}$ are mutually inverse equivalences. 
		\end{enumerate}
	\end{corollary}
	\begin{proof}
			For (1), it is clear that $T_\Lambda^{\Lambda'} P_R(w\!\cdot\!\Lambda)$ is projective with a Verma flag, and Lemma \ref{BGG-reciprocity} shows that the Verma subquotients of $P_R(w\cdot\Lambda)$ are exactly the $M_R(w'\cdot\Lambda)$ with $w'\cdot\Lambda \ge w\cdot\Lambda$, and $M_R(w\cdot\Lambda)$ occurs once. Hence Lemma A.3 shows that any indecomposable summand of $T_\Lambda^{\Lambda'} P_R(w\!\cdot\!\Lambda)$ must be of the form  $P_R({w}'\cdot\Lambda^\prime)$ with ${w}'\cdot\Lambda^\prime \geq {w}\cdot\Lambda^\prime$. By the assumption on $\Lambda, \Lambda^\prime$, the inequality $w'\cdot\Lambda^\prime \geq w\cdot\Lambda^\prime$ forces $w'\cdot\Lambda\geq w\cdot\Lambda$; in particular, $wx\cdot\Lambda\geq w\cdot\Lambda$ for all $x\in \Stab_R(\Lambda^\prime)$. Therefore the summand $P_R({w}\cdot\Lambda^\prime)$ appears with multiplicity $|\Stab_R(\Lambda^\prime)/\Stab_R(\Lambda)|$. 
			For (2), by Lemma \ref{Tran-Verma}, $T_\Lambda^{\Lambda^\prime}M_R(w\cdot\Lambda) \cong M_R(w\cdot\Lambda^\prime)$ and $T_{\Lambda^\prime}^{\Lambda}M_R(w\cdot\Lambda^\prime) \cong M_R(w\cdot\Lambda)$ for all $w\in \Wch_{R,\Lambda}$, which follows that  $T_\Lambda^{\Lambda^\prime}\circ T_{\Lambda^\prime}^{\Lambda}$ and $T_{\Lambda^\prime}^{\Lambda}\circ T_\Lambda^{\Lambda^\prime}$ are naturally isomorphic to the identity functor on the full subcategory of modules admitting a Verma flag; in particular, this holds for projective objects. Since projective modules generate the category and the functors involved are exact, we conclude that they are naturally isomorphic to the identity functor on the entire categories. 
	\end{proof}
	Let $R = \calR$ be as in Subsection \ref{Comb-catO}, and let $\Lambda, \Lambda^\prime\in \rmT_{\bbF}$ be compatible weights satisfying the condition \eqref{WeylCham}.  When $\Lambda$ is regular and $\Stab_R(\Lambda) = \{e, s \}$ for some reflection $s$ in $\Wch_\Lambda$, we refer to the composition $\Theta_{s} :=T_{\Lambda^\prime}^\Lambda\circ T_{\Lambda}^{\Lambda^\prime} $ as the \emph{translation
	through the ``$s$-wall''}. By the above corollary, $\Theta_s$ is independent of the particular choice of $\Lambda, \Lambda^\prime$.
	
	\subsection{Structure functors}\label{App-Structure-functor}Let $R$ be a localization of $U_\bbF^0$ at a prime ideal with $\bbK$ its residue field, and let $X\in \rmT_\bbF/\sim_{R}$.  Suppose that $\Lambda\in X$ is antidominant. Recall the $R$-algebra $\calZ_{R,X}=\End_{\catO_{R}}(P_R(\Lambda))$ is commutative (see Subsection \ref{Deformed-End-Alg}).  The structure functor is defined as 
	\[\xymatrixcolsep{2pc}\xymatrix{
		\bbV = \bbV_{R,X}: = \Hom_{\catO_R}(P_R(\Lambda), \text{-}): \catO_{R,X} \ar[r] &  \calZ_{R,X}\text{-}\mathrm{mod}} \]
	where $\calZ_{R,X}\text{-}\mathrm{mod}$ denotes the category of finitely generated $\calZ_{R,X}$-modules.
	\begin{lemma}\label{propertiesforfunctorV}
		
		\begin{enumerate}[{\rm (1)}]
			\item The functor $\bbV $ is an exact functor, i.e., sends short exact sequences in $\catO_{R,X}$ to short exact sequences in $\calZ_{R,X}\text{-}\mathrm{mod}$. 
			\item For any $w\in \Wch_{R,X}$ there is an isomorphism 
			\[\bbV M_R(w\cdot \Lambda) \cong \calZ_{R,X}/\calI_{w\cdot \Lambda}, \]
			where $\calI_{w\cdot \Lambda}\subset \calZ_{R,X}$ is the ideal generated by elements acting trivially on $\bbV M_R(w\cdot \Lambda)$. 
		\end{enumerate}
	\end{lemma}
	\begin{proof}
		The first assertion is clear. For (2), using the formula \eqref{Proj-basechange} and Lemma \ref{BGG-reciprocity} we obtain that $\dim_\bbK \bbV M_R(\Lambda^\prime)\otimes_R \bbK =1$. 
		Hence, by Nakayama's Lemma, $\bbV(M_R(\Lambda^\prime))$ is generated by a single element as $R$-modules, say $f$. It follows that the map $\calZ_{R,X} \to \bbV M_R(\Lambda^\prime)$ via sending $z$ to $z.f$ is surjective. This proves (2). 
	\end{proof}
	
	Let $\calM_R$ be the full subcategory of $\catO_{R}$ consisting of modules which admits a Verma flag. In particular, all projectives in $\catO_{R}$ belong to $\calM_R$. Let $\calM_{R,X}=\calM_R\cap \catO_{R,X}$ be the corresponding block of $\calM_{R}$. Using the formula \eqref{Proj-basechange} together with Lemma \ref{propertiesforfunctorV}(1) and the argument in the proof of (2), we can directly deduce the following statement.
	\begin{corollary}\label{prop-V}
		\begin{enumerate}[{\rm (1)}]
			\item If $M\in \calM_{R,X}$, then $\bbV M$ is free of finite rank over $R$. In particular, $\bbV M_R(w\cdot \Lambda)$ is free of rank one over $R$. 
			\item $\bbV$ commutes with base changes $R\rightarrow R^\prime$, i.e., there are natural isomorphisms of functors 
			\[\bbV_{R,X}(\text{-})\otimes_R R^\prime \cong \bigoplus_i \bbV_{R^\prime,X_i}(\text{-}\otimes_R R^\prime), \]
			where $X = \sqcup_i X_i$ is the splitting of $X$ under $\sim_{R^\prime}$. \qed 
		\end{enumerate}
	\end{corollary}
	
	When $X$ consists of a single element, the functor $\bbV: \catO_{R, X} \rightarrow R\text{-}\mathrm{mod}$ is an equivalence of categories.
	If $X$ consists of two elements with $\Lambda\in X$ antidominant, let $X^\prime \in \rmT_\bbF/\sim_R$ be the class containing the single element $\Lambda^\prime$. Assume further that $\Lambda$ and $\Lambda^\prime$ are compatible and satisfy the condition \eqref{WeylCham}. Proposition~\ref{EndAlg-Antidom-Projs} then implies that $\calZ_{\calR, X^\prime} \hookrightarrow \calZ_{\calR, X}$.
	Define  $\Ind(\text{-}) = \Ind^{\calZ_{\calR,\Lambda}}_{\calZ_{\calR,\Lambda^\prime}}(\text{-}) =\calZ_{\calR,\Lambda}\otimes_{\calZ_{\calR,\Lambda^\prime}} (\text{-})$ the induction functor, and let $\Res(\text{-})= \Res_{\calZ_{\calR,\Lambda^\prime}}^{\calZ_{\calR,\Lambda}}(\text{-})$ denote the restriction functor.
	Then we have the following results.
	\begin{proposition}\label{Subgeneric-cases}
		Let $X \in \rmT_\bbF/\sim_R$ have two elements with $\Lambda\in X$ antidominant. Then 
		\begin{enumerate}[{\rm (1)}]
			\item For any $M, M^\prime \in \calM_{R,X}$, we have a natural isomorphism 
			\begin{align*}
				\xymatrixcolsep{3pc}\xymatrix{
					\Hom_{\calM_{R,X}}(M, M^\prime)  \ar[r]^-{\widesim{}} &  \Hom_{\calZ_{R,X}\text{-}\mathrm{mod}}(\bbV M, \bbV M^\prime) }
			\end{align*}
			\item Let $X^\prime \in \rmT_\bbF/\sim_R$ consists of one element $\Lambda^\prime$, and $\Lambda, \Lambda^\prime$ are compatible, and satisfy the condition \eqref{WeylCham}. Then there are isomorphisms of functors
			\[\xymatrix{\bbV_{R,X^\prime}\circ T_\Lambda^{\Lambda^\prime}\cong \Res\circ \bbV_{R,X}: \calM_{R,X} \ar[r] & \calZ_{R,X^\prime}\text{-}\mathrm{mod}}  
			\]
			and 
			\[\xymatrix{\bbV_{R,X}\circ T_{\Lambda^\prime}^{\Lambda}\cong \Ind\circ \bbV_{R,X^\prime}: \calM_{R,X^\prime} \ar[r] & \calZ_{R,X}\text{-}\mathrm{mod}.}  
			\]
		\end{enumerate}
	\end{proposition}
	\begin{proof}
		The proof is essentially identical to the arguments for Theorems 5 and 6 in \cite{Fie06}, where the case in which $X$ contains two elements is referred to as \emph{subgeneric}, while $X^\prime$ corresponds to the \emph{generic} case.
	\end{proof}

\end{document}